\documentclass[12pt]{article}
\usepackage{amssymb}
\usepackage{amsfonts}
\usepackage{amsmath}
\usepackage{mathrsfs}
\usepackage{graphicx}
\usepackage{amsbsy}
\usepackage{theorem}
\usepackage{color}
\usepackage{hyperref}
\usepackage[normalem]{ulem}
\def\Section{\setcounter{equation}{0}\section}
\newtheorem{theorem}{Theorem}[section]
\newtheorem{lemma}[theorem]{Lemma}
\newtheorem{corollary}[theorem]{Corollary}
\newtheorem{definition}[theorem]{Definition}
\newtheorem{proposition}[theorem]{Proposition}
\newtheorem{remark}[theorem]{Remark}

\def\thetheorem{\thesection.\arabic{theorem}}
\def\thesection{\arabic{section}}

\def\theequation {\thesection.\arabic{equation}}
\def\beq{\begin{equation}\displaystyle}
\def\eeq{\end{equation}}
\def\bel{\begin{equation} \displaystyle \begin{array}{l} }
\def\eel{\end{array} \end{equation} }
\def\bell{\begin{equation} \displaystyle \begin{array}{ll}  }
\def\eell{\end{array} \end{equation} }

\def\bea{\begin{eqnarray}}
\def\eea{\end{eqnarray} }
\def\bean{\begin{eqnarray*}}
\def\eean{\end{eqnarray*} }
\newenvironment{proof}{\noindent{\bf Proof.~}}
{{\mbox{}\hfill {\small \fbox{}}\\}}
\catcode`@=11
\renewcommand\appendix{\bigskip {\noindent \Large \bf Appendix}
  \setcounter{section}{0}%
  \setcounter{subsection}{0}%
\setcounter{equation}{0}%
\setcounter{theorem}{0}%
\def\thetheorem{A.\arabic{theorem}}
\def\theequation {A.\arabic{equation}}}
\catcode`@=12
\newcommand{\dv}{\mathop{\rm div}\nolimits}

\def\NN{\mathbb{N}}

\def\RR{\mathbb{R}}

\def\ZZ{\mathbb{Z}}

\def\ds{\displaystyle}

\def\bs{\bigskip}

\def\eps{\varepsilon}

\def\pa{\partial}

\def\calM{{\cal M}}
\def\calW{{\cal W}}
\def\calP{{\cal P}}
\def\calD{{\cal D}}
\def\calI{{\cal I}}

\def\smes{{\cal S}_{\cal M}}

\def\achapo{\widehat{a}}
\def\co{\overline{\rm co}}

\def\Xchapo{\widehat{X}}
\def\Zchapo{\widehat{Z}}
\def\nabWchapo{\widehat{\nabla W}}
\def\tilderho{\widetilde{\rho}}
\def\tildeX{\widetilde{X}}
\def\tildea{\widetilde{a}}

\begin{document}

\title{The Filippov characteristic flow for the aggregation equation with mildly singular potentials}

\author{J. A. Carrillo\thanks{Department of Mathematics, Imperial College London, London SW7 2AZ.
Email: \texttt{carrillo@imperial.ac.uk}}, 
F. James \thanks{Math\'ematiques -- Analyse, Probabilit\'es, Mod\'elisation -- Orl\'eans (MAPMO),
Universit\'e d'Orl\'eans \& CNRS UMR 7349}
\thanks{F\'ed\'eration Denis Poisson, Universit\'e d'Orl\'eans \& CNRS FR 2964,
45067 Orl\'eans Cedex 2, France, Email: \texttt{francois.james@univ-orleans.fr}}, 
F. Lagouti\`ere\thanks{Laboratoire de Math\'ematiques, UMR 8628, CNRS -- Universit\'e Paris-Sud, F-91405 Orsay, Email : \texttt{frederic.lagoutiere@math.u-psud.fr}} \& 
N. Vauchelet\thanks{Sorbonne Universit\'es, UPMC Univ Paris 06, UMR 7598, Laboratoire Jacques-Louis Lions, F-75005, Paris, France, Email: \texttt{nicolas.vauchelet@upmc.fr}}
\thanks{CNRS, UMR 7598, Laboratoire Jacques-Louis Lions, F-75005, Paris, France}
\thanks{INRIA-Paris-Rocquencourt, EPC MAMBA, Domaine de Voluceau, BP105, 78153 Le Chesnay Cedex, France} }

\maketitle

\begin{abstract}
Existence and uniqueness of global in time measure solution for the
multidimensional aggregation equation is analyzed. Such a system can be written as a 
continuity equation with a velocity field computed through a self-consistent interaction potential.
In Carrillo et al. (Duke Math J (2011)) \cite{Carrillo}, a well-posedness theory based on the geometric approach of gradient flows in measure metric spaces has been developed for mildly singular potentials at the origin under the basic assumption of being $\lambda$-convex.
We propose here an alternative method using classical tools from PDEs. We 
show the existence of a characteristic flow based on Filippov's theory of discontinuous 
dynamical systems such that the weak measure solution is the pushforward measure with this flow.
Uniqueness is obtained thanks to a contraction argument in transport distances using the $\lambda$-convexity of the potential. Moreover, we show the equivalence of this solution with the
gradient flow solution. Finally, we show the convergence of a numerical scheme for general measure solutions in this framework allowing for the simulation of solutions for initial smooth densities after their first blow-up time in $L^p$-norms.
\end{abstract}

\bs



\bs

\Section{Introduction}

This paper is devoted to the so-called aggregation equation in $d$ space dimension
    \beq\label{EqInter}
\pa_t\rho = \dv\big((\nabla_x W*\rho) \rho\big) , \qquad t>0,\quad x\in\RR^d,
    \eeq
complemented with the initial condition $\rho(0,x)=\rho^{ini}$. Here, $W$ plays the role of an interaction potential whose gradient $\nabla_x W(x-y)$ measures the relative force exerted by an infinitesimal mass localized at a point $y$ onto an infinitesimal mass located at a point x.

This system appears in many applications in physics and population dynamics.
In the framework of granular media, equation \eqref{EqInter} is used to
describe the large time dynamics of inhomogeneous kinetic models (see \cite{benedetto,CCV,Toscani}).
Model of crowd motion with a nonlinear dependancy of the term $\nabla_xW*\rho$
are also encountered in \cite{pieton,pieton2}.
In population dynamics, \eqref{EqInter} provides a biologically meaningful description
of aggregative phenomena. The description of the collective migration of cells by swarming
leads to such non-local interaction PDEs (see e.g. \cite{morale,okubo,topaz}).
Another example is the modelling of bacterial chemotaxis. In this framework, the quantity $S=W*\rho$
is the chemoattractant concentration which is a substance emitted by bacteria allowing them to
interact with each others. The dynamics can be macroscopically modelled by the Patlak-Keller-Segel system
\cite{keller,patlack}. In the kinetic framework, the Othmer-Dunbar-Alt model is usually used,
its hydrodynamic limit leads to the aggregation equation \eqref{EqInter} \cite{dolschmeis,filblaurpert,NoDEA}.
In many of these examples, the potential $W$ is usually mildly singular, i.e. $W$ has a weak singularity
at the origin. Due to this weak regularity, finite time blow-up of regular solutions
has been observed for such systems and has gained the attention of several authors
(see e.g. \cite{Li,BV,Bertozzi1,Bertozzi2,Carrillo}). Finite time concentration is sometimes considered as a very 
simple mathematical way to mimick aggregation of individuals, as opposed to diffusion. Finally, attraction-repulsion 
potentials have been recently proposed as very simple models of pattern formation due to the rich structure of the 
set of stationary solutions, see \cite{PSTV,BU,BUKB,BCLR,BBSKU} for instance.

Since finite time blow-up of regular solutions occurs, a natural framework to
study the existence of global in time solutions is to work in the space of probability measures.
However, several difficulties appear due to the weak regularity of the potential.
In fact, the definition of the product of $\nabla W*\rho$ with $\rho$
is a priori not well defined. This fact has already been noticed in one dimension
in \cite{NoDEA,GF_dual}. Using defect measures in a two-dimensional framework,
existence of weak measure solutions for parabolic-elliptic coupled system
has been obtained in \cite{PoupaudDef,DolbSchm}. However, uniqueness is lacking.
Measure valued solutions for the 2D Keller-Segel system have been considered
in \cite{LuckSugiVela} as limit of solutions of a regularized problem.

For the aggregation equation \eqref{EqInter}, a well-posedness theory for measure
valued solutions has been considered using the geometrical approach
of gradient flows in \cite{Carrillo}. This technique has been extended to
the case with two species in \cite{DiFrancesco}. The assumptions on the potential in order
to get this well-posedness theory of measure valued solutions use certain convexity of the potential
that allows for mild singularity of the potential at the origin.

In this paper, we assume that the interaction potential $W\,:\,\RR^d\to\RR$ satisfies the following properties:
\begin{itemize}
\item[{\bf (A0)}] $W$ is Lipschitz continuous, $W(x)=W(-x)$ and $W(0)=0$.
\item[{\bf (A1)}] $W$ is $\lambda$-convex for some $\lambda\leq 0$, i.e.
  $W(x)-\frac{\lambda}{2}|x|^2$ is convex.
\item[{\bf (A2)}] $W\in C^1(\RR^d\setminus\{0\})$.
\end{itemize}
\textcolor{blue}{This set of potentials includes the class of so-called {\it pointy} potentials,
which have a pointy tip at the origin.}
A typical example is a fully attractive Morse type potential, $W(x)=1-e^{-|x|}$, which is $-1$-convex.

Let us emphasize that we only consider Lipschitz potentials which allows to bound the velocity field,
whereas in \cite{Carrillo}, linearly growing at infinity potentials are allowed.
In other words, we assume that there exists a nonnegative constant $w_\infty$ such that for all $x\neq 0$,
\beq\label{borngradW}
|\nabla W(x)| \leq w_\infty.
\eeq
The main reason for this restriction is to be able to work with suitable characteristics for this velocity field as explained below.

Denoting $a=-\nabla W*\rho$ the macroscopic velocity, equation \eqref{EqInter}
can be considered as a conservative transport equation with velocity field $a$. Then
a traditional definition for solutions is the one defined thanks to
the characteristics corresponding to this macroscopic velocity.
However, the velocity $a$ is not Lipschitz and therefore we cannot
defined classical solutions to the characteristics equation.
To overcome this difficulty, Filippov \cite{Filippov} has proposed
a notion of solution which extend the classical one. Using
this so-called Filippov flow, Poupaud \& Rascle \cite{PoupaudRascle} have proposed a notion
of solution to the conservative linear transport equation defined by
$X_\# \rho^{ini}$ where $X$ is the Filippov flow corresponding
to the macroscopic velocity. However a stability result of the flow
was still lacking until recently \cite{Bianchini}, and thus 
there are no results with this technique for nonlinear equations of the form \eqref{EqInter}.
We notice that in one dimension and for linear equations, these solutions are equivalent to
the duality solutions defined in \cite{bj1,BJpg}, which have been successfully
used in \cite{NoDEA, GF_dual} to tackle \eqref{EqInter} in the one dimensional case.

On the other hand, although the geometric approach of gradient flows furnishes a
general framework for well-posedness, this approach does not allow to define a characteristic 
flow corresponding to the macroscopic velocity $a=-\nabla W*\rho$. In this work, we focus on
improving the understanding of these solutions by showing that under assumptions  {\bf (A0)}-{\bf (A2)} 
on the potential, the solutions can be understood as their initial data pushed forward by suitable characteristic flows. 

In order to achieve this goal, we first generalize the theory developed in \cite{PoupaudRascle} 
to the nonlinear aggregation equation \eqref{EqInter}. The first difficulty is, as it was in 
\cite{Carrillo}, to identify the right definition of the nonlinear term and the nonlinear product. 
This was solved in \cite{Carrillo} by identifying the element of minimal norm by subdifferential calculus. 
We revisit this issue by clarifying that this is the right definition of the nonlinear term if we approximate 
a pointy potential by smooth symmetric potentials. Once the identification of the right velocity field has been done, 
we use the crucial stability results of Filippov's flows in \cite{Bianchini} to pass to the limit in the nonlinear 
terms. This leads to the construction of global measure solutions of the form $X_\# \rho^{ini}$, where $X$ is the 
Filippov flow associated to the velocity vector field $a$. This is the point where we need globally bounded velocity 
vector fields since Filippov's theory \cite{Filippov} was only developed under these assumptions. 
In this way, we extend to the muti-dimensional case the results in \cite{NoDEA} 
(for a particular choice of the potential $W$) and in \cite{GF_dual}.

Moreover, we are able to adapt arguments for uniqueness already used for the aggregation equation and for nonlinear 
continuity equations as in \cite{Loeper,CR,Bertozzi2} to show the contraction property of the Wasserstein distance 
for our constructed solutions. This leads to a uniqueness result for our constructed solutions and to show the 
equivalence between the notion of gradient flow solutions and these Filippov's flow characteristics solutions.  
Let us further comment that in the one dimensional case it has been noticed that there is a link between solutions 
to \eqref{EqInter} and entropy solutions to scalar conservation law for an antiderivative of $\rho$ 
(see \cite{BV,bonaschi, NoDEA, GF_dual}). This link has allowed to consider 
extensions of the model \eqref{EqInter} with a nonlinear dependency of the term $\nabla W*\rho$.

Finally, let us mention that apart from particle methods to the aggregation equations, very few numerical schemes 
have been proposed to simulate solutions of the aggregation equation after blow-up. The so-called sticky particle 
method was shown to be convergent in \cite{Carrillo} and used to obtain qualitative properties of the solutions such 
as the finite time total collapse. However, this method is not that practical to deal with finite time blow-up and 
the behavior of solutions after blow-up in dimensions larger than one. In one dimension, such numerical simulations thanks 
to a particle scheme have been obtained by part of the authors in \cite{NoDEA}. Moreover, in the one dimensional case and with a 
nonlinear dependency of the term $\nabla W*\rho$, they propose in \cite{sinum} a finite volume scheme allowing to simulate the 
behaviour after blow up and prove its convergence.
Finally, extremely accurate numerical schemes have been developed to study the blow-up profile for smooth solutions, 
see \cite{Huang1,Huang2}. In fact, part of the authors recently proposed an energy decreasing finite volume method 
\cite{CCH} for a large class of PDEs including in particular \eqref{EqInter} but no convergence result was given. 
Here, we give a convergence result for a finite volume scheme and for general measures as initial data. 
This allows for numerical simulations of solutions in dimension greater than one allowing to observe the behaviour after blow-up occurs.

The outline of the paper is the following. Next section is devoted to the definition of our notion of weak
measure solutions for the aggregation equation. After introducing some notations, we first recall the basic 
results as obtained by Poupaud \& Rascle \cite{PoupaudRascle} on measure solutions
for conservative linear transport equations. Then we define the notion of solutions
defined by a flow and state the main result of this paper  in Theorem \ref{Exist}.
Finally, we recall the existence result of gradient flow solutions in \cite{Carrillo}
and state their equivalence with solutions defined by a flow.
Section \ref{sec:proof} is devoted to the proof of the existence and
uniqueness result. The main ingredient of the proof of existence
is a one-sided Lipschitz property of the macroscopic velocity and
an atomization strategy by approximating with finite Dirac Deltas. 
A contraction argument in Wasserstein distance for these solutions allows to recover the uniqueness.
In Section \ref{sec:num}, we investigate the numerical approximation
of such solutions. A finite volume scheme is proposed and its convergence
is established for general measure valued solutions. An illustration thanks 
to numerical simulations is also provided showing the ability of the scheme to capture the finite 
time total collapse and the qualitative interaction between different aggregates after the first 
blow-up in $L^p$-norms. Finally, an Appendix is devoted to some technical Lemmas useful throughout
the paper.

%%%%%%%%%%%%%%%%%%%%%%%%%%%%%%%%%%%%%%%%%%%%%%%%%%%%%%%%%%%%%%%%%%%%%%
\Section{Weak measure solutions for the aggregation equation}
\label{sec:defsol}
%%%%%%%%%%%%%%%%%%%%%%%%%%%%%%%%%%%%%%%%%%%%%%%%%%%%%%%%%%%%%%%%%%%%%%

All along the paper, we will make use of the following notations.
We denote $\calM_{loc}(\RR^d)$ the space of locally finite measures on $\RR^d$. For $\rho\in {\cal M}_{loc}(\RR^d)$, we denote 
by $|\rho|(\RR^d)$ its total variation. We denote $\calM_b(\RR^d)$ the space of measures in $\calM_{loc}(\RR^d)$ with finite total variation.
From now on, the space of measures $\calM_b(\RR^d)$ is
always endowed with the weak topology $\sigma({\cal M}_b,C_0)$.
For $T>0$, we denote $\smes :=C([0,T];{\cal M}_b(\RR^d)-\sigma({\cal M}_b,C_0))$.
Finally, we define the space of probability measures with finite
second order moment:
    $$
\calP_2(\RR^d) = \left\{\mu \mbox{ nonnegative Borel measure}, \mu(\RR^d)=1, \int |x|^2 \mu(dx) <\infty\right\}.
    $$
This space is endowed with the Wasserstein distance $d_W$ defined by (see e.g. \cite{Villani1, Villani2})
    \beq\label{defWp}
d_{W}(\mu,\nu)= \inf_{\gamma\in \Gamma(\mu,\nu)} \left\{\int |y-x|^2\,\gamma(dx,dy)\right\}^{1/2}
    \eeq
where $\Gamma(\mu,\nu)$ is the set of measures on $\RR^d\times\RR^d$ with marginals $\mu$ and $\nu$, i.e.
\begin{align*}
\Gamma(\mu,\nu) = \left\{ \gamma\in \calP_2(\RR^d\times\RR^d); \ \forall\, \xi\in C_0(\RR^d), \right. &
\int \xi(y_1)\gamma(dy_1,dy_2) = \int \xi(y_1) \mu(dy_1), \\
& \left.\int \xi(y_2)\gamma(dy_1,dy_2) = \int \xi(y_2) \nu(dy_2) \right\}.
\end{align*}
From a minimization argument, we know that in the definition of $d_{W}$
the infimum is actually a minimum.
A map that realizes the minimum in the definition \eqref{defWp}
of $d_{W}$ is called an {\it optimal plan}, the set of which is denoted by $\Gamma_0(\mu,\nu)$.
Then for all $\gamma_0\in \Gamma_0(\mu,\nu)$, we have
$$
d_{W}^2(\mu,\nu)= \int |y-x|^2\,\gamma_0(dx,dy).
$$

%%%%%%%%%%%%%%%%%%%%%%%%%%%%%%%%%%%

\subsection{Weak measure solutions for conservative transport equation}

We recall in this Section some useful results
on weak measure solutions to the conservative transport equation
\beq\label{CTE}
\pa_t u + \dv(b u)=0; \qquad u(t=0) = u^0.
\eeq
We assume here that the vector field $b$ is given.

We start by the following definition of characteristics \cite{Filippov}~:
\begin{definition}\label{DefFlow}
Let us assume that $b=b(t,x)\in \RR^d$ is a vector field defined on $[0,T]\times \RR^d$ with $T>0$.
A Filippov characteristic $X(t;s,x)$ stems from $x\in \RR^d$ at time $s$
is a continuous function $X(\cdot;s,x) \in C([0,T],\RR^d)$ such that
$\frac{\pa}{\pa t} X(t;s,x)$ exists a.e. $t\in [0,T]$ satisfying
	$$
\frac{\pa}{\pa t} X(t;s,x) \in \big\{ {\mbox Convess} (b)(t,\cdot)\big\}
(X(t;s,x)) \quad \mbox{a.e. } t\in [0,T]; \qquad
X(s;s,x) = x.
	$$
From now on, we will use the notation $X(t,x)=X(t;0,x)$.
\end{definition}

In this definition $Convess(E)$ denotes the essential convex hull of
a set $E$. We remind the reader the definition for the sake of completeness, 
see \cite{Filippov,AubinCellina} for more details. We denote by $Conv(E)$ the classical convex
hull of $E$, i.e., the smallest closed convex set containing $E$. Given the vector field 
$b(t,\cdot):\RR^d\longrightarrow\RR^d$, the essential convex hull at point $x$ is defined as
	$$
\{ {\mbox Convess} (b)(t,\cdot)\}(x)=\bigcap_{r>0} \bigcap_{N\in \mathcal{N}_0} Conv\left[ b\left( t, B(x,r)\setminus N\right)\right]\,,
	$$
where $\mathcal{N}_0$ is the set of zero Lebesgue measure sets.
Then, we have the following existence and uniqueness result of Filippov
characteristics under the mere assumption that the vector field $b$
is one-sided Lipschitz.
	\begin{theorem}[\cite{Filippov}]
Let $T>0$. Let us assume that the vector field
$b\in L^1_{loc}(\RR;L^\infty(\RR^d))$ satisfies the OSL condition, that is
for all $x$ and $y$ in $\RR^d$, for all $t\in [0,T]$,
\beq\label{OSLC}
(b(t,x)-b(t,y))\cdot(x-y) \leq \alpha(t) \|x-y\|^2,
\quad\mbox{ for }\alpha \in L^1(0,T).
\eeq
Then there exists an unique Filippov characteristic $X$
associated to this vector field.
	\end{theorem}

An important consequence of this result is the existence and uniqueness
of weak measure solutions for the conservative linear transport equation.
This result has been proved by Poupaud and Rascle \cite{PoupaudRascle}.
\begin{theorem}[\cite{PoupaudRascle}]
Let $T>0$.
Let $b\in L^1([0,T],L^\infty(\RR^d))$ be a vector field satisfying
the OSL condition \eqref{OSLC}. Then for any $u_0\in \calM_b(\RR^d)$,
there exists a unique measure solution $u$ in $\smes$ to the conservative
transport equation \eqref{CTE} such that $u(t)=X(t)_\#u_0$, where $X$ is
the unique Filippov characteristic, i.e. for any $\phi\in C_0(\RR^d)$, we have
$$
\int_{\RR^d} \phi(x) u(t,dx) = \int_{\RR^d} \phi(X(t,x)) u_0(dx),
\qquad \mbox{ for } t\in [0,T].
$$
\end{theorem}

Finally, we recall the following stability result for the Filippov
characteristics which has been established by Bianchini and Gloyer \cite[Theorem 1.2]{Bianchini}
	\begin{theorem}\label{th:bianchini}
Let $T>0$. Assume that the sequence of vector fields $b_n$ converges weakly to $b$ in $L^1([0,T],L^1_{loc}(\RR^d))$.
Then the Filippov flow $X_n$ generated by $b_n$ converges locally in
$C([0,T]\times\RR^d)$ to the Filippov flow $X$ generated by $b$.
	\end{theorem}

%%%%%%%%%%%%%%%%%%%%%%%%%%%%%%%%%%%%%%%%%%%%%%%%%%%%%%%%%%%%%%%%%%%%%
\subsection{Solutions defined by Filippov's flow}

We state in this Section the main result of this paper dealing
with the existence and uniqueness of measure solutions defined
thanks to the Filippov characteristics for the aggregation equation \eqref{EqInter}.
For $\rho\in C([0,T],\calP_2(\RR^d))$,
we define the velocity field $\widehat{a}_\rho$ by
\beq\label{achapo}
\achapo_\rho(t,x) = -\int_{y\neq x} \nabla W(x-y) \rho(t,dy).
\eeq
This choice of macroscopic velocity will be justified by the convergence result
of Lemma \ref{lemstab_a} below. We remark that this definition of the velocity field coincides 
with the one based on subdifferential calculus done in \cite{Carrillo}, see next subsection.
Due to the $\lambda$-convexity of $W$ {\bf (A1)}, we deduce that for all $x$, $y$ in $\RR^d\setminus \{0\}$
we have
\beq\label{lambdaconv}
(\nabla W(x)-\nabla W(y))\cdot (x-y) \geq \lambda \|x-y\|^2.
\eeq
For the sake of simplicity of the notations, we introduce
$$
\nabWchapo(x) = \left\{
\begin{array}{ll} \nabla W(x), \qquad & \mbox{ for } x\neq 0; \\
0, & \mbox{ for } x=0, \end{array}
\right.
$$
such that, by definition of the velocity \eqref{achapo}, we have
\beq\label{achaposym}
\achapo_{\rho}(t,x) = -\int_{\RR^d} \nabWchapo(x-y) \rho(t,dy)\,.
\eeq
Moreover, since $W$ is even, $\nabla W$ is odd and by taking $y=-x$ in
\eqref{lambdaconv}, we deduce that inequality \eqref{lambdaconv} is
true even when $x$ or $y$ vanishes for $\nabWchapo$~:
\beq\label{lambdaconvWchapo}
\forall\, x,y\in\RR^d, \qquad
(\nabWchapo(x)-\nabWchapo(y))\cdot (x-y) \geq \lambda \|x-y\|^2.
\eeq

We are now ready to state the main result of this paper. Its proof is postponed until Section \ref{sec:proof} below.
	\begin{theorem}\label{Exist}
Let $W$ satisfy assumptions {\bf (A0)--(A2)} and let $\rho^{ini}$
be given in $\calP_2(\RR^d)$. Given $T>0$, there exists a unique Filippov characteristic flow $X$ such that
the pushforward measure $\rho:=X_\# \rho^{ini}$ is a distributional
solution of the aggregation equation
\beq\label{eq:distrib}
\pa_t \rho + \dv(\achapo_\rho \rho) = 0, \qquad \rho(0,\cdot)=\rho^{ini},
\eeq
where $\achapo_\rho$ is defined by \eqref{achapo}.

Besides, if $\rho^{ini}$ and $\mu^{ini}$ are two given nonnegative
measure in $\calP_2(\RR^d)$, then the corresponding pushforward measures
$\rho$ and $\mu$ satisfy for all $t\in [0,T]$
\beq\label{ineqcontraction}
d_W(\rho(t),\mu(t)) \leq e^{-2\lambda t} d_W(\rho^{ini},\mu^{ini}).
\eeq
% Moreover, we have $\rho\in C([0,T],\calP_2(\RR^d))$ and
% the energy estimate for all $0\leq t_0 \leq t_1 \leq T$~:
% \beq\label{energy=}
% \int_{t_0}^{t_1} \int_{\RR^d} |\achapo_\rho(t,x)|^2\rho(t,dx) + \calW(\rho(t_1))
% =\calW(\rho(t_0)).
% \eeq
	\end{theorem}

\begin{remark}
Let us point out that the exponent in the stability estimate in $d_W$ in \eqref{ineqcontraction} can be 
improved to $-\lambda t$ if both initial measures $\rho^{ini}$ and $\mu^{ini}$ have the same center of mass.
\end{remark}

%%%%%%%%%%%%%%%%%%%%%%%%%%%%%%%%%%%%%%%%%%%%%%%%%%%%%%%%%%%%%%%%%%%%
\subsection{Gradient flow solutions}

We recall the definition of gradient flow solutions as defined in \cite{Ambrosio,Carrillo}.
Let $\calW$ be the energy of the system defined by
\beq\label{nrjinter}
\calW(\rho)= \frac 12\int_{\RR^d\times\RR^d} W(x-y)\,\rho(dx)\rho(dy).
\eeq
We say that $\mu\in AC^2_{loc}([0,+\infty);\calP_2(\RR^d))$ if $\mu$ is locally H\"older continuous of exponent 
$1/2$ in time with respect to the distance $d_W$ in $\calP_2(\RR^d)$.

\begin{definition}[Gradient flows]\label{defgradflow}
%(cf Definition 11.1.1 of \cite{Ambrosio}).
Let $W$ satisfy assumptions {\bf (A0)--(A2)}. We say that a map $\mu\in AC^2_{loc}([0,+\infty);\calP_2(\RR^d))$
is a solution of a gradient flow equation associated to the functional
$\calW$, defined in \eqref{nrjinter}, if there exists a Borel vector field $v$ such that
$v(t)\in Tan_{\mu(t)}\calP_2(\RR^d)$ for a.e. $t>0$,
$\|v(t)\|_{L^2(\mu)} \in L^2_{loc}(0,+\infty)$, the continuity equation
$$
\pa_t \mu + \dv \big(v \mu \big)=0,
$$
holds in the sense of distributions, and $v(t) = -\pa^0 \calW(\mu(t))$ for a.e.
$t>0$. Here $\pa^0\calW(\mu)$ denotes the element of minimal norm in $\pa\calW(\mu)$, 
which is the subdifferential of $\calW$ at the point $\mu$.
\end{definition}

We refer to \cite{Ambrosio,Carrillo} for details about the definition of the subdifferential 
since we will not make use of them in the sequel. The existence and uniqueness result of 
\cite[Theorem 2.12 and 2.13 ]{Carrillo} can now be synthetized as follows.

    \begin{theorem}[\cite{Carrillo}]
\label{GradFlow}
Let $W$ satisfy assumptions {\bf (A0)--(A2)}.
Given $\rho^{ini}\in \calP_2(\RR^d)$,
there exists a unique gradient flow solution of \eqref{EqInter},
i.e. a curve $\rho_{GF}\in AC_{loc}^2([0,\infty);\calP_2(\RR^d))$ satisfying
    $$
\begin{array}{l}
\ds \frac{\pa \rho_{GF}(t)}{\pa t} + \dv (v(t)\rho_{GF}(t))=0, \qquad \mbox{in }\calD'([0,\infty)\times \RR^d), \\
\ds v(t,x)=-\pa^0 \calW(\rho_{GF})(t,x) = -\int_{y\neq x} \nabla W(x-y)\,\rho_{GF}(t,dy),
\end{array}
    $$
with $\rho_{GF}(0)=\rho^{ini}$. Moreover, the following energy identity holds for all $0\leq t_0\leq t_1 < \infty$:
    \begin{equation*}%\label{EstimEnerg}
\int_{t_0}^{t_1} \int_{\RR^d} |\partial^0W*\rho_{GF}|^2 \rho_{GF}(t,dx)dt + \calW(\rho_{GF}(t_1)) = \calW(\rho_{GF}(t_0)).
    \end{equation*}
\end{theorem}

Theorems \ref{Exist} and \ref{GradFlow} furnish two notions of solutions
to \eqref{EqInter} which are solutions in the sense of distributions.
Then we should wonder on the link between this two notions.
The following result states their equivalence.

\begin{theorem}\label{equivalence}
Let $W$ satisfy assumptions {\bf (A0)--(A2)}.
Let $\rho^{ini}\in \calP_2(\RR^d)$ be given.
Let us denote $\rho$ the solution of Theorem \ref{Exist} and
by $\rho_{GF}$ the solution of Theorem \ref{GradFlow}.
Then we have $\rho\in AC_{loc}^2([0,\infty);\calP_2(\RR^d))$
and $\rho=\rho_{GF}$.
\end{theorem}

As a consequence of this equivalence result, there exists a unique 
solution $\rho$ which satisfies in the sense of distribution \eqref{eq:distrib}
with $\achapo_\rho$ defined in \eqref{achapo}. This solution is 
a pushforward measure by a characteristic flow: $\rho=X_\# \rho^{ini}$.

%%%%%%%%%%%%%%%%%%%%%%%%%%%%%%%%%%%%%%%%%%%%%%%%%%%%%%%%%%%%%%%%%%%%%%
\Section{Existence and uniqueness}\label{sec:proof}
%%%%%%%%%%%%%%%%%%%%%%%%%%%%%%%%%%%%%%%%%%%%%%%%%%%%%%%%%%%%%%%%%%%%%%

\subsection{Macroscopic velocity and one-sided estimate}

In order to justify the choice of the expression of the macroscopic velocity in \eqref{achapo},
we prove a stability result for symmetric potentials. Moreover, we state in Lemma \ref{aOSL}
the important one-sided Lipschitz property for this macroscopic velocity.

\begin{lemma}\label{lemstab_a}
Let us assume that $W$ satisfies assumptions {\bf (A0)--(A2)}.
Let $(W_n)_{n\in\NN^*}$ be a sequence of even functions in $C^1(\RR^d)$ satisfying {\bf (A1)}
and \eqref{borngradW} with the same constants $\lambda$ and $w_\infty$ not depending on $n$ and such that
\beq\label{boundnabWn}
sup_{x\in \RR^d\setminus B(0,\frac 1n)} \big|\nabla W_n(x)-\nabla W(x)\big| \leq \frac 1n, \qquad \mbox{ for all } n\in \NN^*.
\eeq
If the sequence $\rho_n\rightharpoonup \rho$ weakly as measures, then
for every continuous compactly supported $\phi$, we have
$$
\lim_{n\to +\infty} \iint_{\RR^d\times \RR^d} \phi(x) \nabla W_n(x-y) \rho_n(dx)\rho_n(dy)
=\iint_{\RR^d\times \RR^d \setminus D} \phi(x) \nabla W(x-y) \rho(dx)\rho(dy),
$$
where $D$ is the diagonal in $\RR^d$: $D=\{(x,x),\, x\in \RR^d\}$.
\end{lemma}

\begin{proof}
The construction of such an approximating sequence of potentials can be obtained for instance using the Moreau-Yosida regularization, see \cite{Ambrosio} and \cite[Proposition 3.5]{CLM}. Let us focus on the last property. We first notice that by symmetry of $W_n$, we have for all $\phi\in Lip(\RR^d)$,
$$
\int_{\RR^d} \phi(x) a_n(x) \rho_n(dx) = \frac 12
\iint_{\RR^d\times \RR^d} (\phi(x)-\phi(y)) \nabla W_n (x-y) \rho_n(dx)\rho_n(dy).
$$
We recall that since $\rho_n\rightharpoonup \rho$ weakly as measures, we have that
$\rho_n\otimes \rho_n \rightharpoonup \rho\otimes \rho$ weakly as measures.
Let $\eps>0$. Since $\phi$ is continuous on a compact set, it is uniformly continuous
therefore there exists $\alpha>0$ such that $|\phi(x)-\phi(y)|\leq \eps$ for $|x-y|\leq \alpha$.
Then, defining $D_\alpha=\{ (x,y)\in \RR^d\times\RR^d, \  |x-y|<\alpha\}$ for any $\alpha>0$,
we split the latter integral into~:
$$
\begin{array}{l}
\ds \iint_{\RR^d\times \RR^d} (\phi(x)-\phi(y)) \Big(\nabla W_n (x-y) \rho_n(dx)\rho_n(dy)
-\nabWchapo(x-y) \rho(dx)\rho(dy)\Big) = \\[2mm]
\ds \qquad\qquad \iint_{\RR^d\times \RR^d\setminus D_\alpha} (\phi(x)-\phi(y)) \Big(\nabla W_n (x-y) \rho_n(dx)\rho_n(dy)
-\nabWchapo(x-y) \rho(dx)\rho(dy)\Big) \\[2mm]
\ds \qquad\qquad +\iint_{D_\alpha} (\phi(x)-\phi(y)) \Big(\nabla W_n (x-y) \rho_n(dx)\rho_n(dy)
-\nabWchapo(x-y) \rho(dx)\rho(dy)\Big).
\end{array}
$$
For the last term of the right hand side, we use the fact that $\phi$ is uniformly continuous
and \eqref{borngradW} for $W$ and $W_n$ to prove that
$$
\iint_{D_\alpha} (\phi(x)-\phi(y)) \Big(\nabla W_n (x-y) \rho_n(dx)\rho_n(dy)
-\nabWchapo(x-y) \rho(dx)\rho(dy)\Big) \leq C \eps.
$$
For the first term, we have
$$
\begin{array}{l}
\ds \iint_{\RR^d\times \RR^d\setminus D_\alpha} (\phi(x)-\phi(y)) \Big(\nabla W_n (x-y) \rho_n(dx)\rho_n(dy)
-\nabWchapo(x-y) \rho(dx)\rho(dy)\Big) = \\[2mm]
\ds \qquad\qquad\qquad \iint_{\RR^d\times \RR^d\setminus D_\alpha} (\phi(x)-\phi(y))
\big(\nabla W_n (x-y)- \nabWchapo(x-y)\big) \rho_n(dx)\rho_n(dy)  \\[2mm]
\ds \qquad\qquad\qquad +\iint_{\RR^d\times \RR^d\setminus D_\alpha} (\phi(x)-\phi(y))
\nabWchapo(x-y) \big(\rho_n(dx)\rho_n(dy)-\rho(dx)\rho(dy)\big).
\end{array}
$$

Using \eqref{boundnabWn} we deduce that the first term of the right hand side is bounded by
$\eps$ for $n$ large enough. For the second term, we use the fact that
$(x,y)\mapsto (\phi(x)-\phi(y))\nabWchapo(x-y)$ is continuous and compactly supported and
the tight convergence of $\rho_n$ towards $\rho$ to prove it is bounded by $\eps$ when
$n$ is large enough.
This concludes the proof.
\end{proof}

\begin{remark}
In other words, this Lemma states that if $W_n$ is an approximating smooth
and even sequence for $W$ and
for any sequence $\rho_n$ converging to $\rho$ in $\smes$, then,
denoting $a_n=\nabla W_n*\rho_n$, we have the convergence of the flux
$a_n \rho_n \rightharpoonup \achapo_\rho \rho$ in the weak topology $\smes$
with $\achapo_\rho$ defined in \eqref{achapo}.
A similar convergence result has been proved in \cite{PoupaudDef},
although in this paper, the potential is less regular
and in particular it does not satisfies {\bf (A0)} neither the bound \eqref{borngradW}.
Then at the limit the author recovers a defect measure which vanishes in our case.
Such result has also been used in \cite{DolbSchm} to define weak solution for
the two-dimensional Keller-Segel system for chemotaxis.
\end{remark}

\begin{lemma}\label{aOSL}
Let $\rho(t)\in \calM_b(\RR^d)$ be nonnegative such that
$|\rho(t,\cdot)|(\RR^d)\leq c$ for all $t\geq 0$.
Then under assumptions {\bf (A0) -- (A2)} the function
$(t,x)\mapsto \achapo_\rho(t,x)$ defined in \eqref{achapo}
or equivalently in \eqref{achaposym}
satisfies the one-sided Lipschitz (OSL) estimate
\beq\label{achapoOSL}
(\achapo_\rho(t,x)-\achapo_\rho(t,y))\cdot(x-y) \leq -\lambda |\rho|(\RR^d) \|x-y\|^2.
\eeq
\end{lemma}

\begin{proof}
This result is an easy consequence of the $\lambda$-convexity of the potential.
In fact, by definition \eqref{achaposym}, we have
	$$
\achapo_\rho(x)-\achapo_\rho(y) = -\int_{\RR^d} \big(\nabWchapo(x-z)
-\nabWchapo(y-z)\big)\rho(dz).
	$$
Using inequality \eqref{lambdaconvWchapo} and the nonnegativity of $\rho$, we readily obtain \eqref{achapoOSL}.
\end{proof}

From Lemma \ref{aOSL}, we deduce that
if $\rho\in C([0,T],\calP_2(\RR))$ and $\achapo_\rho$ is defined as
in \eqref{achapo}, we can define the Filippov characteristic flow,
denoted $\Xchapo$, associated to the velocity field $\achapo_\rho$
(see \cite{Filippov}). Then we consider the push-forward measure
$$
\rho_{PR} := \Xchapo_\# \rho^{ini}.
$$
Poupaud \& Rascle \cite{PoupaudRascle} have shown that this measure is
the unique measure solution of the conservative linear transport equation
$$
\pa_t \rho_{PR} + \dv (\achapo_\rho \rho_{PR})=0.
$$
The difficulty here is that the measure $\rho$ used in the definition of
the macroscopic velocity $\achapo_\rho$ is a priori not the same as $\rho_{PR}$. Actually,
the whole aim of the next subsection is to prove that they are equal.

%%%%%%%%%%%%%%%%%%%%%%%%%%%%%%%%%%%%%%%%%%%%%%%%%%%%%%%%%%%%%%%
\subsection{Existence}\label{sec:exist}

In this subsection, we prove the existence part of Theorem \ref{Exist}. We follow the idea of atomization consisting in approximating the solution by a finite sum of Dirac masses or particles, and then passing to the limit. This approach has been very successful for the aggregation equation, see \cite{Bertozzi1,Carrillo,GF_dual,bonaschi} for instance.

\medskip
{\bf Approximation with Dirac masses}.
Let us assume that the initial density is given by
$\rho^{ini,N}(x) = \sum_{i=1}^N m_i \delta(x-x_i^0)$, with $x_i^0\neq x_j^0$ for $i\neq j$, for
a finite integer $N$ and belongs to $\calP_2(\RR^d)$, i.e. we have
\beq\label{hyprhoini}
\sum_{i=1}^N m_i = 1, \qquad M_2(0):=\sum_{i=1}^N m_i |x_i^0|^2  <+\infty.
\eeq
Then we look for a solution of the aggregation equation given by
$$
\rho^N(t,x) = \sum_{i=1}^N m_i \delta(x-x_i(t)).
$$
By definition \eqref{achapo} we have
$$
\achapo_{\rho^N} (t,x)= \left\{
\begin{array}{ll}
\ds - \sum_{j=1}^N m_j \nabla W (x-x_j(t))\ , \qquad
& \ds \mbox{ if } x\neq x_i, i=1,\ldots,N, \\[4mm]
\ds - \sum_{j\neq i} m_j \nabla W(x_i(t)-x_j(t))\ , & \mbox{ otherwise.}
\end{array}\right.
$$
For such a macroscopic velocity, we can define the Filippov characteristic
$\Xchapo^N$ as in Definition \ref{DefFlow}.
In fact, from Lemma \ref{aOSL}, $\achapo_{\rho^N}$ satisfies the OSL condition,
which allows to define uniquely the Filippov characteristic.
It is obvious from the essential convex hull definition that
$$
- \sum_{j\neq i} m_j \nabla W(x_i(t)-x_j(t)) \in \{\mbox{Convess}
(\achapo_{\rho^N})(t,\cdot)\} (x_i(t)).
$$
Then setting the classical ODE system $x'_i(t) = -\sum_{j\neq i} m_j \nabla W(x_i(t)-x_j(t))$, 
the solution will be defined up to the time $t_c$ of the first collision between two or more particles. 
By uniqueness of the Filippov characteristic, $\Xchapo^N(t,x_i^0)=x_i(t)$ until that time. At time $t_c$, 
one has to recompute the velocity field, since the colliding particles will stick together for later times according to the rule given by
$$
\frac{\pa}{\pa t} \Xchapo^N(t;s,x) \in \{ {\mbox Convess} (\achapo_{\rho^N})(t,\cdot)\}
(\Xchapo^N(t;s,x)) \quad \mbox{a.e. } t\in [0,T]; \qquad
\Xchapo^N(s;s,x) = x.
$$
This construction of the characteristics coincides with the one done in \cite[Remark 2.10]{Carrillo}. 
In other words, the Filippov flow coincides with this time evolution+collision+gluing of particles procedure.

Next, we define $\rho_{PR}^N = \Xchapo^N_{\ \#} \rho^{ini,N}$.
By construction, this measure satisfies in the sense of distributions
$$
\pa_t \rho_{PR}^N + \dv \big(\achapo_{\rho^N} \rho_{PR}^N\big) = 0.
$$
Moreover, from the definition of the pushforward measure, we can write
$$
\achapo_{\rho_{PR}^N} = -\int_{\RR^d} \nabWchapo(x-y) \rho_{PR}^N(dy)
= - \int_{\RR^d} \nabWchapo(x-\Xchapo^N(t,y)) \rho^{ini,N}(dy).
$$
By definition of $\rho^{ini,N}$, we deduce
$$
\begin{array}{ll}
\ds \achapo_{\rho_{PR}^N}(t,x) & \ds = -\sum_{i=1}^N m_i \int_{\RR^d}
\nabWchapo(x-\Xchapo^N(t,y)) \delta(y-x_i^0) \\[2mm]
&\ds = - \sum_{i=1}^N m_i \nabWchapo(x-\Xchapo^N(t,x_i^0))
= \achapo_{\rho^N}(t,x).
\end{array}
$$
Thus we conclude that $\rho_{PR}^N=\rho^N$.

Let us consider now the bound on the second moment.
We define $M_2^N(t):=\sum_{i=1}^N m_i |x_i(t)|^2$. Differentiating, we have
$$
\frac{d}{dt} M_2^N(t) = 2\sum_{i=1}^N \sum_{j=1}^N m_i m_j x_i \nabWchapo(x_i-x_j).
$$
Using \eqref{borngradW}, we deduce that
$$
\frac{d}{dt} M_2^N(t) \leq 2C \sum_{i=1}^N \sum_{j=1}^N m_i m_j |x_i|.
$$
From the Cauchy-Schwarz inequality and the fact that $\sum_i m_i=1$,
we deduce
\beq\label{boundM2N}
\frac{d}{dt} M_2^N(t) \leq K(1+M_2^N(t)).
\eeq
Since $M_2^N(0)$ is finite from \eqref{hyprhoini},
we deduce from a Gronwall Lemma that for all $t\in [0,T]$
we have $M_2^N(t)<+\infty$.
By continuity of the Filippov flow, we have that
$\rho^N\in C([0,T],\calP_2(\RR^d))$.
Moreover, using \eqref{borngradW}, we deduce that
\beq\label{boundachapo}
|\achapo_{\rho^N}(t,x)|\leq C.
\eeq

\medskip
{\bf Passing to the limit $N\to +\infty$}.
Let us assume that $\rho^{ini}\in \calP_2(\RR^d)$ and consider an
approximation $\rho^{ini,N}\in \calP_2(\RR^d)$ given by a finite sum of Dirac masses such
that $\rho^{ini,N} \rightharpoonup \rho^{ini}$ weakly in the sense of measures in $\calM_b(\RR)$ as $N\to +\infty$ 
with a uniform in $N$ bound of the second moments, or equivalently, $d_W(\rho^{ini,N},\rho^{ini})\to 0$ as $N\to\infty$.
We have proved above that we can construct a Filippov flow $\Xchapo^N$
and a measure
$\rho^N = \Xchapo^N\,_{\#} \rho^{ini,N}\in C([0,T],\calP_2(\RR^d))$
such that in the distributional sense
	$$
\pa_t \rho^N + \dv (\achapo_{\rho^N} \rho^N) = 0,
	$$
where $\achapo_{\rho^N}$ is defined by \eqref{achapo}.
From \eqref{boundachapo}, we have that
$\achapo_{\rho^N}$ is bounded in $L^\infty([0,T]\times\RR^d)$.
Thus $\achapo_{\rho^N}$ converges up to a subsequence towards $b$
in $L^\infty_{t,x}-weak*$.  We can pass to the limit in the distributional sense in
the one-sided Lipschitz inequality \eqref{achapoOSL} satisfied by
$\achapo_{\rho^N}$, since the right hand side of this inequality
does not depend on $N$.
Then $b$ satisfies the OSL condition and we can
define $X_b$ the Filippov flow corresponding to $b$.
From the $L^\infty_{t,x}-weak*$ convergence above, it is obvious that $\achapo_{\rho^N}$
converges weakly to $b$ in $L^1([0,T];L^1_{loc}(\RR^d))$. Therefore, we can apply Theorem \ref{th:bianchini},
and deduce that $\Xchapo^N \to X_b$ locally in $C([0,T]\times \RR^d)$ as $N\to +\infty$.

Moreover, for every $\phi \in C_0(\RR^d)$, we have
	$$
\int_{\RR^d} \phi(x) \rho^N(t,dx)
= \int_{\RR^d} \phi(\Xchapo^N(t,x)) \rho^{ini,N}(dx).
	$$
Since $\rho^{ini,N} \rightharpoonup \rho^{ini}$ weakly in the sense of measures and
$\Xchapo^N(t,x) \to X_b(t,x)$ locally in $C([0,T]\times\RR^d)$,
we deduce that for any $R>0$,
	$$
\lim_{N\to +\infty} \int_{B(0,R)}  \phi\big(\Xchapo^N(t,x)\big) \,\rho^{ini,N}(dx)
= \int_{B(0,R)}  \phi(X_b(t,x)) \,\rho^{ini}(dx).
	$$
Denoting as above $M_2^N (0)$ (resp. $M_2(0)$) the second order moment of $\rho^{ini,N}$  (resp. $\rho^{ini}$), we infer that
$$
\int_{\RR^d\setminus B(0,R)} \rho^{ini,N}(dx) \leq \frac{M_2^N(0)}{R^2}\leq \frac{d_W(\rho^{ini,N},\rho^{ini})+M_2(0)}{R^2}.
$$
This implies that for all $\phi\in C_0(\RR^d)$,
$$
\int_{\RR^d} \phi(\Xchapo^N(t,x)) \rho^{ini,N}(dx) \underset{N \to +\infty}{\longrightarrow}
\int_{\RR^d} \phi(X_b(t,x)) \rho^{ini}(dx) = \int_{\RR^d} \phi(x)\, X_b\,_\#\rho^{ini}(dx).
$$
We deduce that $\rho^N \rightharpoonup \rho:=X_b\,_\# \rho^{ini}$
in $\smes$ as $N\to +\infty$.
Finally, from this latter convergence, we deduce by applying Lemma \ref{lemA}
that $\achapo_{\rho^N} \to \achapo_{\rho}$ a.e.
By uniqueness of the limit, we conclude that $b=\achapo_{\rho}$ a.e.

\medskip
{\bf Bound on the second moment}.
Finally, we recover the bound in $\calP_2(\RR^d)$. We first notice that due to the approximation of the initial 
data done in the previous step, we know that $M_2^N(0)$ is bounded uniformly in $N$.
Taking into account this fact together with \eqref{boundM2N}, there exists a nonnegative constant $C_T$
depending only on $T$ and the initial data $\rho^{ini}$ such that
	$$
M_2^N(t)=\int_{\RR^d} |\Xchapo^N(t,x)|^2\,\rho^{ini,N}(dx) \leq C_T\,.
	$$
Then $|x|^2 \rho^N(t)$ is a bounded sequence of nonnegative measures that converges weakly as measures to $|x|^2\rho(t)$. 
Therefore, by the Banach-Alaoglu theorem, we get
	$$
M_2(t)=\int_{\RR^d} |X_b(t,x)|^2\,\rho^{ini}(dx) \leq \liminf_{N\to \infty} M_2^N(t) \leq C_T\,.
	$$
This ends the proof of existence.

%%%%%%%%%%%%%%%%%%%%%%%%%%%%%%%%%%%%%%%%%%%%%%%%%%%%%%%%%%%%%%%
\subsection{Uniqueness}
\label{sec:uniq}

The proof of the uniqueness relies on a contraction property with respect
to the Wasserstein distance $d_W$. In the framework of general gradient flows,
this property has been established using the $\lambda$-geodesically convexity of the energy
in \cite[Theorem 11.1.4]{Ambrosio}, see also \cite{CCV,Loeper,CR,Carrillo} for
related results. We show here an equivalent result for our notion of solution.
The proof relies strongly on the definition of the solution as a
pushforward measure associated to a flow and on the $\lambda$-convexity
of $W$.

\begin{proposition}\label{prop:contraction}
Let us assume that $W$ satisfies assumptions {\bf (A0) -- (A2)}.
Let $\rho_0$ and $\tilderho_0$ be two nonnegative measure in
$\calP_2(\RR^d)$.
Let $\rho$ and $\tilderho$ in $C([0,T],\calP_2(\RR^d))$ be solutions of
the aggregation equation as in Theorem \ref{Exist}
with initial data $\rho_0$ and $\widetilde{\rho_0}$ respectively.
Then for all $t>0$,
$$
d_W(\rho(t),\tilderho(t)) \leq e^{-2\lambda t} d_W(\rho_0,\tilderho_0)\,.
$$
Moreover if $\rho_0$ and $\widetilde{\rho_0}$ have the same center of mass, then for all $t>0$,
$$
d_W(\rho(t),\tilderho(t)) \leq e^{-\lambda t} d_W(\rho_0,\tilderho_0)\,.
$$
\end{proposition}

\begin{proof}
Let $\rho_0$ and $\tilderho_0$ be two nonnegative measure in $\calP_2(\RR^d)$.
We first choose an optimal plan $\gamma_0\in \Gamma_0(\rho_0,\tilderho_0)$
such that we have
	$$
d_W^2(\rho_0,\tilderho_0)=\iint_{\RR^d\times\RR^d} |x_1-x_2|^2 \,\gamma_0(dx_1,dx_2).
	$$
We regularize the potential $W$, as in Lemma \ref{lemstab_a}, by $W_\eps \in C^1(\RR^d)$ such that
$W_\eps$ is $\lambda$-convex, $W_\eps(-x)=W_\eps(x)$,
$|\nabla W_\eps|\leq |\nabla W|$ and
	$$
\sup_{x\in \RR^d\setminus B(0,\eps)} \big|\nabla W_\eps(x)- \nabla W(x)| \leq \eps.
	$$

As in subsection \ref{sec:exist}, we construct a Filippov flow $X_\eps$
associated to the velocity field $a_\eps :=-\int_{\RR^d} \nabla W_\eps(x-y) \rho_\eps(t,dy)$
such that $\rho_\eps=X_\eps\,_\# \rho_0 \,\in C([0,T],\calP_2(\RR^d))$ is a measure solution to the aggregation
equation
$\pa_t\rho_\eps + \dv(a_\eps\rho_\eps)=0$
with initial data $\rho_0$.
For this flow we have
	$$
\frac{d}{dt} X_\eps (t,x) = -\int_{\RR^d} \nabla W_\eps(x-y) \rho_\eps(t,dy);
\qquad X_\eps(0,x)=x.
	$$
Similarly we construct $\tilderho_\eps = \tildeX_\eps\,_\#\tilderho_0
\in C([0,T],\calP_2(\RR^d))$ associated to the velocity field
$\tildea_\eps :=-\int_{\RR^d} \nabla W_\eps(x-y) \tilderho_\eps(t,dy)$.

By definition of the pushforward measure, we have that
$$
a_{\eps}(t,x) = -\int_{\RR^d} \nabla W_\eps(x-X_\eps(t,y)) \rho_0(dy),\quad
\tildea_\eps(t,x) = -\int_{\RR^d} \nabla W_\eps(x-\widetilde{X}_\eps(t,y)) \tilderho_0(dy).
$$
Moreover, from the definition of the optimal plan $\gamma_0$ we can rewrite
\begin{eqnarray}
\label{aeps1}
&  \ds a_\eps(t,x) = -\iint_{\RR^d\times\RR^d} \nabla W_\eps(x-X_\eps(t,y_1)) \,\gamma_0(dy_1,dy_2),  \\[3mm]
\label{aeps2}
&  \ds \tildea_{\eps}(t,x) = -\iint_{\RR^d\times\RR^d} \nabla W_\eps(x-\tildeX_\eps(t,y_2)) \,\gamma_0(dy_1,dy_2).
\end{eqnarray}
Since $\rho_\eps(t)$ belongs to $\calP_2(\RR^d)$, we have that
$$
\int_{\RR^d} |x|^2 \rho_\eps(t,dx) = \int_{\RR^d} |X_\eps(t,x)|^2 d\rho_0(x) <\infty.
$$
The same estimate holds true for $\tilderho_\eps$.
Then we can consider the quantity
	$$
\calI_\eps(t) = \iint_{\RR^d\times\RR^d} \big| X_\eps(t,x_1)-\tildeX_\eps(t,x_2) \big|^2 \,\gamma_0(dx_1,dx_2).
	$$
We notice that for $t=0$, we have $\calI(0)=d_W^2(\rho_0,\tilderho_0)$.
We have
	$$
\frac{d}{dt} \calI_\eps = 2\iint_{\RR^d\times\RR^d} \big(a_\eps(t,X_\eps(t,x_1)) -
\tildea_{\eps}(t,\tildeX_\eps(t,x_2))\big) \cdot(X_\eps(t,x_1) -\tildeX_\eps(t,x_2)) \,\gamma_0(dx_1,dx_2).
	$$

From the definition of the velocity field \eqref{aeps1}--\eqref{aeps2}, we have
	$$\begin{array}{ll}
\ds \frac{d}{dt} \calI_\eps = - 2\iiiint_{(\RR^d)^4} & \ds
 \big(\nabla W_\eps(X_\eps(t,x_1)-X_\eps(t,y_1))-\nabla W_\eps (\tildeX_\eps(t,x_2)-\tildeX_\eps(t,y_2)) \big)\cdot \\[2mm]
& \ds (X_\eps(t,x_1)-\tildeX_\eps(t,x_2)) \,\gamma_0(dx_1,dx_2) \gamma_0(dy_1,dy_2).
	\end{array}$$
From assumption $W_\eps(-x)=W_\eps(x)$, we deduce that $\nabla W_\eps$ is odd.
Then $\nabla W_\eps(X_\eps(t,x)-X_\eps(t,y))=-\nabla W_\eps(X_\eps(t,y)-X_\eps(t,x))$ for all $x$, $y$.
By exchanging the role of $(x_1,x_2)$ and $(y_1,y_2)$ in this latter
equality and using the symmetry of $\nabla W_\eps$ we deduce that
	$$\begin{array}{ll}
\ds \frac{d}{dt} \calI_\eps = 2\iiiint_{(\RR^d)^4}  & \ds
\big(\nabla W_\eps(X_\eps(t,x_1)-X_\eps(t,y_1)) -\nabla W_\eps(\tildeX_\eps(t,x_2)-\tildeX_\eps(t,y_2)) \big)\cdot  \\[2mm]
&\ds (X_\eps(t,y_1)-\tildeX_\eps(t,y_2)) \,\gamma_0(dx_1,dx_2) \gamma_0(dy_1,dy_2).
	\end{array}$$
Summing these two latter equalities, we obtain
	$$\begin{array}{ll}
\ds \frac{d}{dt} \calI_\eps = -\iiiint_{(\RR^d)^4}  & \ds
\big(\nabla W_\eps(X_\eps(t,x_1)-X_\eps(t,y_1)) -\nabla W_\eps(\tildeX_\eps(t,x_2)-\tildeX_\eps(t,y_2)) \big)\cdot  \\[2mm]
&\ds (X_\eps(t,x_1)-X_\eps(t,y_1)-\tildeX_\eps(t,x_2)+\tildeX_\eps(t,y_2)) \,\gamma_0(dx_1,dx_2) \gamma_0(dy_1,dy_2).
	\end{array}$$
From the $\lambda$-convexity of $W$, we deduce from \eqref{lambdaconv} that
\begin{equation}\label{tech1}
\frac{d}{dt} \calI_\eps \leq -\lambda \iiiint_{(\RR^d)^4}
\big|X_\eps(t,x_1)-X_\eps(t,y_1)-\tildeX_\eps(t,x_2)+\tildeX_\eps(t,y_2)\big|^2 \,\gamma_0(dx_1,dx_2) \gamma_0(dy_1,dy_2).
\end{equation}
We recall that $\lambda\leq 0$ and $|\rho_0|(\RR^d)=|\tilderho_0|(\RR^d)=1$.
A direct Young inequality leads to
\beq\label{gronwallI}
\frac{d}{dt} \calI_\eps \leq -4\lambda \calI_\eps.
\eeq
Applying the Gronwall lemma, we deduce that
\beq\label{ineqIeps}
\calI_\eps(t) \leq e^{-4\lambda t} \calI(0) = e^{-4\lambda t} d_W^2(\rho_0,\tilderho_0).
\eeq
If the initial data have the same center of mass, then it is easy to check that the center of mass remains
the same for both solutions for all times, that is, for all $t\geq 0$
\begin{align*}
M_1=&\,\int_{\RR^d} x \,\rho_\eps(t,dx)  = \int_{\RR^d} X_\eps(t,x) \,\rho_0(dx)  = \iint_{(\RR^d)^2} X_\eps(t,x_1) \,\gamma_0(dx_1,dx_2) \\
=&\,\int_{\RR^d} y \,\tilderho_\eps(t,dy)  = \int_{\RR^d} \tildeX_\eps(t,y) \,\tilderho_0(dy)  = \iint_{(\RR^d)^2} \tildeX_\eps(t,x_2) \,\gamma_0(dx_1,dx_2)\,.
\end{align*}
Thus, one can check that
	$$
\iiiint_{(\RR^d)^4}
(X_\eps(t,x_1)-\tildeX_\eps(t,x_2)) \cdot (X_\eps(t,y_1)-\tildeX_\eps(t,y_2)) \,\gamma_0(dx_1,dx_2) \gamma_0(dy_1,dy_2)=0\,.
	$$
Now, expanding the square in \eqref{tech1}, we improve the decay by a factor of 2 in \eqref{gronwallI} getting
	$$
\frac{d}{dt} \calI_\eps \leq -2\lambda \calI_\eps.
	$$

From now on, we stick to the general case to pass to the limit $\eps\to 0$ in \eqref{ineqIeps}.
Since $\rho_\eps(t)$ is bounded in $\calP_2(\RR^d)$ independently on $\eps$,
we deduce from the Prokhorov theorem that
we can extract a subsequence such that $\rho_\eps(t) \rightharpoonup \rho(t)$
tightly. Then applying Lemma \ref{lemB} in the Appendix we deduce that
$a_\eps \to \achapo_\rho$ for a.e. $t\in [0,T]$, $x\in \RR^d$,
where $\achapo_\rho$ is defined in \eqref{achapo}.
Then we have shown that we can construct a Filippov characteristic flow $X$
associated to the velocity field $\achapo_\rho$.
Applying the stability result of \cite{Bianchini}, recalled in Theorem \ref{th:bianchini},
we deduce that $X_\eps \to X$ locally in $C([0,T]\times \RR^d)$. We can proceed analogously for $\tilderho(t)$, and thus, for any $R>0$ we have
$$
\lim_{\eps\to 0} \int_{B(0,R)} |X(t,x)-X_\eps(t,x)|^2 \rho_0(dx) =
\lim_{\eps\to 0} \int_{B(0,R)} |\tildeX(t,x)-\tildeX_\eps(t,x)|^2 \tilderho_0(dx) = 0\,.
$$
We conclude that
\begin{equation}\label{tech2}
\iint_{B(0,R)\times B(0,R)} \left[\big| X_\eps(t,x_1)-\tildeX_\eps(t,x_2) \big|^2 - \big| X(t,x_1)-\tildeX(t,x_2) \big|^2\right] \,\gamma_0(dx_1,dx_2) \longrightarrow 0
\end{equation}
as $\eps\to 0$.

Now, using \eqref{ineqIeps} together with \eqref{tech2}, we deduce
$$
\iint_{B(0,R)\times B(0,R)} \big| X(t,x_1)-\tildeX(t,x_2) \big|^2 \,\gamma_0(dx_1,dx_2) \leq e^{-4\lambda t} d_W^2(\rho_0,\tilderho_0)\,,
$$
for all $R>0$, leading to our final desired estimate
\beq\label{ineqI}
\calI(t):=\iint_{\RR^d\times\RR^d} |X(t,x_1)-\tildeX(t,x_2)|^2\, \gamma_0(dx_1,dx_2)
\leq e^{-4\lambda t} d_W^2(\rho_0,\tilderho_0).
\eeq
Finally, by definition of the Wasserstein distance \eqref{defWp}, we deduce
$d_W^2(\rho,\tilderho)\leq \calI(t)$ and the contraction inequality
\eqref{ineqcontraction} follows directly.
\end{proof}

The uniqueness of solution in Theorem \ref{Exist} is then a trivial consequence
of this contraction property.
In fact, applying Proposition \ref{prop:contraction} for two solutions
$\rho$ and $\tilderho$ with the same initial data $\rho^{ini}$,
we deduce from \eqref{ineqI} that $X=\tildeX$ on $supp (\rho^{ini})$
which implies that $\rho=\tilderho$.

%%%%%%%%%%%%%%%%%%%%%%%%%%%%%%%%%%%%%%%%%%%%%%%%%%%%%%%%%%%%%%%
\subsection{Equivalence with gradient flow solutions}

This subsection is devoted to the proof of the equivalence of solution defined
by the Filippov flow with the gradient flow solution as stated in Theorem \ref{equivalence}.
For $\rho^{ini}$ given in $\calP_2(\RR^d)$, we denote
$\rho$ the solution of Theorem \ref{Exist} and
$\rho_{GF}$ the solution of Theorem \ref{GradFlow}.
We have proved above the existence of a Filippov characteristic flow $X$
such that $\rho=X_\#\rho^{ini}$ and $\rho$ satisfies in the sense of distributions
$$
\pa_t \rho + \dv (\achapo_{\rho} \rho) = 0.
$$
From the bound on $\achapo_{\rho}$ in \eqref{boundachapo}, we deduce
since $\rho$ belongs to $C([0,T],\calP_2(\RR^d))$ that
$\achapo_{\rho}$ is bounded in $L^2([0,T],L^2(\rho(t)))$.
Thus using Theorem 8.3.1 of \cite{Ambrosio}, we deduce that
$\rho\in AC^2([0,T],\calP_2(\RR^d))$.
We can conclude that $\rho$ is a gradient flow solution,
see \cite[Sections 8.3 and 8.4]{Ambrosio} and \cite{Carrillo}.
We conclude the proof using the uniqueness of gradient
flow solutions. As a consequence, the solutions constructed in Theorem \ref{Exist}
satisfy the energy identity, for all $0\leq t_0\leq t_1< \infty$,
$$
\int_{t_0}^{t_1} \int_{\RR^d} |\achapo_\rho(t,x)|^2\rho(t,dx) + \calW(\rho(t_1))
=\calW(\rho(t_0)).
$$

%%%%%%%%%%%%%%%%%%%%%%%%%%%%%%%%%%%%%%%%%%%%%%%%%%%%%%%%%%%%%%%%%%%%%%%%%%
\Section{Numerical approximation}
\label{sec:num}

This Section is devoted to the convergence of a numerical scheme for simulating solutions given by Theorem \ref{Exist}. The theory of existence developed in the previous section will allow to prove convergence of standard finite volume schemes, provided the discretized macroscopic velocity is accurately
defined. Before that, we would like to comment on particle schemes.

%%%%%%%%%%%%%%%%%%%%
\subsection{Particle scheme}

The contraction estimate in $d_W$ for solutions leads to a theoretical estimate of the convergence error of the particle scheme used in the first step of the proof of Theorem \ref{Exist}. This was already pointed out in \cite{Carrillo} in the framework of gradient flow solutions and used for qualitative behavior properties. We just remind the main result here for completeness. Let us consider an initial distribution given by a finite sum of $N$ Dirac masses $\rho^{ini,N}=\sum_{i=1}^N m_i\delta(x-x_i^0)$.
We consider the sticky particles dynamics given by
$$
x'_i(t)= - \sum_{j\neq i} m_j \nabla W(x_i(t)-x_j(t)), \qquad
x_i(0)=x_i^0, \qquad i=1,\ldots,N.
$$
These dynamics are well defined provided $x_i(t)\neq x_j(t)$. When two or more particles meet, we stick them and the resulting system follows the same dynamics with one or more particle less. This system of ODEs plus the collision+gluing particle procedure gives the solution $\rho^N(t)=\sum_{i=1}^N m_i\delta(x-x_i(t))$ of Theorem \ref{Exist} at time $t\geq 0$ with initial data $\rho^{ini,N}$ as explained in the first step of its proof.

\begin{corollary}\label{cor2}
Let $\rho^{ini}\in \calP_2(\RR^d)$, we denote $\rho\in C([0,T];\calP_2(\RR^d))$
the corresponding solution in Theorem \ref{Exist} with initial data $\rho^{ini}$.
Let $\rho^{ini,N}$ be given in $\calP_2(\RR^d)$ by $\rho^{ini,N}(x)=\sum_{i=1}^N m_i\delta(x-x_i^0)$
an approximation such that $d_W(\rho^{ini},\rho^{ini,N}) \to 0$ as
$N\to +\infty$. Given $T>0$, then the corresponding solution $\rho^N$ with initial data $\rho^{ini,N}$ defined above verifies
$$
\sup_{t\in [0,T]} d_W(\rho(t),\rho_N(t)) \underset{N \to +\infty}{\longrightarrow} 0.
$$
\end{corollary}

The previous corollary is a direct consequence of the stability property in Theorem \ref{Exist}. Although this result is very nice from the theoretical viewpoint, it is not that useful for simulating the evolution of
equation \eqref{EqInter} for fully attractive potentials in practice. The reason is twofold. On one hand, to get a good control on the error after a long time one needs a very large number of particles. On the other hand, the treatment of the collision between particles and the gluing procedure is not too difficult in one dimension but it is very cumbersome (and difficult to control its error) in more dimensions. Nevertheless, particle simulations lead to a very good understanding of qualitative properties of solutions for attractive-repulsive potentials where collisions do not happen, see \cite{BU,BUKB,BCLR,BBSKU} for instance.
We finally mention the recent result of convergence of smooth particle schemes toward smooth solutions of the aggregation equation before blow-up in \cite{CB}.

%%%%%%%%%%%%%%%%%%%%%%%%%%%%%%%%%%%%%%%%%%%%%%
\subsection{Finite volume discretization}

In the next three subsections, we will concentrate on the convergence of a finite volume scheme for the solutions constructed in Theorem \ref{Exist} with general measures as initial data. The one dimensional case has been considered in \cite{sinum}. This case is particular since we can
define an antiderivative of the measure solution $\rho$ which is then a BV function solution
of an equation obtained by integrating the aggregation equation. This fact is very much connected to the relation of the one dimensional case with conservation laws as in \cite{BV,NoDEA,bonaschi,GF_dual}.
Then the convergence of the numerical scheme relies on a TVD property. We refer the reader to
\cite{sinum} for more details such as the importance of a {\it good} choice of the macroscopic velocity
which is emphasized with some numerical examples.

We focus in this work to higher dimensions where such techniques
cannot be applied. For the sake of clarity, we restrict ourselves
to the case $d=2$. We consider a cartesian grid $x_i=i\Delta x$
and $y_j=j\Delta y$, for $i\in \ZZ$ and $j\in \ZZ$. We denote by
$C_{ij}$ the cells $C_{ij}=[x_i,x_{i+1})\times[y_j,y_{j+1})$. The
time discretization is given by $t_n=n\Delta t$, $n\in \NN$. As
usual, we denote $\rho_{ij}^n$ an approximation of
$\rho(t_n,x_i,y_j)$. We consider that the potential $W$ is given
and satisfies assumptions {\bf (A0)-(A2)}.

Following the idea in \cite{sinum}, we propose the following
discretization. For a given nonnegative measure $\rho^{ini}\in
\calP_2(\RR^2)$, we define for $i,j\in \ZZ^2$, \beq\label{disrho0}
\rho_{ij}^0=\frac{1}{\Delta x\Delta y}\iint_{C_{ij}}
\rho^{ini}(dx,dy)\geq 0. \eeq Since $\rho^{ini}$ is a probability
measure, the total mass of the system is $\sum_{i,j}
\rho_{ij}^0\Delta x\Delta y=1$. Assuming that an approximating
sequence $(\rho_{ij}^n)_{i,j}$ is known at time $n$, then we
compute the approximation at time $t_{n+1}$ by~:
\beq\label{dis_num}
\begin{array}{ll}
\ds \rho_{ij}^{n+1} = & \ds \rho_{ij}^n - \frac{\Delta t}{\Delta
x}\big({a_x}^n_{i+1/2j} \rho_{i+1/2 j}^n - {a_x}^n_{i-1/2j}
\rho_{i-1/2 j}^n\big) - \frac{\Delta t}{\Delta
y}\big({a_y}^n_{ij+1/2} \rho_{i j+1/2}^n -
{a_y}^n_{ij-1/2} \rho_{i j-1/2}^n\big) \\[2mm]
& \ds + \frac{\Delta t}{2\Delta x} w_\infty
\big(\rho_{i+1j}^n-2\rho_{ij}^n+\rho_{i-1j}^n\big) + \frac{\Delta
t}{2\Delta y} w_\infty
\big(\rho_{ij+1}^n-2\rho_{ij}^n+\rho_{ij-1}^n\big),
\end{array}
\eeq
where $w_\infty$ is defined in \eqref{borngradW}.
We have used the notation
$$
\begin{array}{ll}\ds \rho_{i+1/2j} = \frac{\rho_{ij}+\rho_{i+1j}}{2},
& \ds \rho_{ij+1/2} = \frac{\rho_{ij}+\rho_{ij+1}}{2}, \\[2mm]
\ds {a_x}_{i+1/2j} = \frac{{a_x}_{ij}+{a_x}_{i+1j}}{2}, \qquad
& \ds {a_y}_{ij+1/2} = \frac{{a_y}_{ij}+{a_y}_{ij+1}}{2}.
\end{array}
$$
The macroscopic velocity is defined by
\begin{equation}
\label{def:ai12j} {a_x}_{ij} =  \frac{1}{\Delta x\Delta y}
\sum_{k,\ell} \rho_{k\ell} \,D_xW_{ij}^{k\ell}, \qquad {a_y}_{ij}
= \frac{1}{\Delta x\Delta y} \sum_{k,\ell} \rho_{k\ell}
\,D_yW_{ij}^{k\ell},
\end{equation}
where
$$
\begin{array}{l}
\ds D_xW_{ij}^{k\ell} := \iint_{C_{k\ell}} \Big(\iint_{C_{ij}} \widehat{\pa_xW} \big(x-x',y-y' \big) \,dxdy\Big)dx'dy', \\[5mm]
\ds D_yW_{ij}^{k\ell} := \iint_{C_{k\ell}} \Big(\iint_{C_{ij}} \widehat{\pa_yW} \big(x-x',y-y' \big) \,dxdy\Big)dx'dy'.
\end{array}
$$
We notice after a straightforward change of variable that we have also
\begin{equation}
\label{ai12jbis} {a_x}_{i+1/2j} = \frac{1}{\Delta x\Delta y}
\sum_{k,\ell} \rho_{k+1/2\ell} \,D_xW_{ij}^{k\ell},\qquad
{a_y}_{ij+1/2} = \frac{1}{\Delta x\Delta y} \sum_{k,\ell}
\rho_{k\ell+1/2} \,D_yW_{ij}^{k\ell}.
\end{equation}

Let us finally remark that this scheme is close to the Lax-Friedrichs flux formula for conservation laws. Therefore, it introduces some numerical viscosity in the simulations. This will be clear in the error terms obtained in the convergence proof since we will have error estimates depending on second order derivatives, see subsection 4.4.

%%%%%%%%%%%%%%%%%%%%%%%%%%%%%%%%%%%%%%%%%%%%%%%%%%%%%%%%%%%%%%%%%%%%%%%%%
\subsection{Properties of the scheme}

The following Lemma states a CFL-like condition for the scheme~:
\begin{lemma}\label{lem:CFL}
Let us assume that $W$ satisfies {\bf (A0)-(A2)}
and consider $\rho^{ini}\in \calP_2(\RR^2)$.
We define $\rho_{ij}^0$ by \eqref{disrho0}.
Let us assume that the condition
\begin{equation}\label{CFL}
w_\infty\Big(\frac{1}{\Delta x}+\frac{1}{\Delta y}\Big)\Delta t
\leq \frac{1}{2},
\end{equation}
is satisfied. Then the sequences computed thanks to the scheme defined in \eqref{dis_num}--\eqref{def:ai12j}
satisfy for all $i$, $j$ and $n$,
$$
\rho_{ij}^n \geq 0, \qquad |{a_x}_{ij}^n|\leq w_\infty, \qquad |{a_y}_{ij}^n|\leq w_\infty.
$$
\end{lemma}
\begin{proof}
The total initial mass of the system is $\Delta x\Delta
y\sum_{i,j} \rho_{ij}^0=1$. Since the scheme \eqref{dis_num} is
conservative, we have for all $n\in \NN$, $\Delta x\Delta y
\sum_{i,j} \rho_{ij}^n=1$.

We can rewrite equation \eqref{dis_num} as
\begin{align}
\rho_{ij}^{n+1} = & \,\rho_{ij}^n \left[ 1- \frac{\Delta
t}{\Delta x} \left(\frac{{a_x}^n_{i+1/2j}-{a_x}^n_{i-1/2j}}{2}\right)
-\frac{\Delta t}{\Delta y}
\left(\frac{{a_y}^n_{ij+1/2}-{a_y}^n_{ij-1/2}}{2}\right)
-\frac{\Delta t}{\Delta x} w_\infty - \frac{\Delta t}{\Delta y} w_\infty \right]
\nonumber\\[3mm]
& +  \rho_{i+1j}^n \frac{\Delta t}{2\Delta x}\Big(w_\infty -
{a_x}^n_{i+1/2j}\Big)
+ \rho_{i-1j}^n \frac{\Delta t}{2\Delta x}\Big(w_\infty + {a_x}^n_{i-1/2j}\Big)  
\nonumber\\[3mm]
& +  \rho_{ij+1}^n \frac{\Delta t}{2\Delta y}\Big(w_\infty -
{a_y}^n_{ij+1/2}\Big) + \rho_{ij-1}^n \frac{\Delta t}{2\Delta
y}\Big(w_\infty + {a_y}^n_{ij-1/2}\Big). \label{schemarho}
\end{align}

Let us prove by induction on $n$ that for all $i,j,n$ we have $\rho_{ij}^n \geq 0$.
Let us assume that for a given $n\in \NN$ we have $\rho_{ij}^n \geq 0$ for all $i,j$.
Then, from definition \eqref{def:ai12j} and assumption \eqref{borngradW}
we clearly have that
$$
|{a_x}_{ij}^{n}|\leq w_\infty \Delta x\Delta y \sum_{i,j}
\rho_{ij}^n = w_\infty\quad ; \qquad |{a_y}_{ij}^n|\leq w_\infty.
$$
Then assuming that the condition \eqref{CFL} holds, we deduce that in the scheme \eqref{schemarho}
all the coefficients in front of $\rho_{ij}^n$, $\rho_{i-1j}^n$, $\rho_{i+1j}^n$, $\rho_{ij-1}^n$,
and $\rho_{ij+1}^n$ are nonnegative.
Thus, using the induction assumption, we deduce that
$\rho_{ij}^{n+1}\geq 0$ for all $i,j$.
\end{proof}

In the following Lemma, we gather some properties of the scheme: mass conservation,
center of mass conservation and finite second order moment.

\begin{lemma}\label{bounddismom}
Let us assume that $W$ satisfies {\bf (A0)-(A2)}
and consider $\rho_{ij}^0$ defined by \eqref{disrho0} for some
$\rho^{ini}\in \calP_2(\RR^2)$.
Let us assume that \eqref{CFL} is satisfied.
Then the sequence $(\rho_{ij}^n)$ constructed thanks to the numerical scheme
\eqref{dis_num}--\eqref{def:ai12j} satisfies:

$(i)$ Mass conservation and conservation of the center of mass:
for all $n\in \NN^*$, we have
\begin{eqnarray*}
%\label{massconserv} 
&\ds \sum_{i,j\in \ZZ^2} \rho_{ij}^n \Delta x
\Delta y =
\sum_{i,j\in \ZZ^2} \rho_{ij}^0 \Delta x \Delta y = 1\ , \\
%\label{centermassconserv}
&\ds \sum_{i,j\in \ZZ^2} x_i \rho_{ij}^n = \sum_{i,j\in \ZZ^2} x_i \rho_{ij}^0\ , \qquad
\sum_{i,j\in \ZZ^2} y_j \rho_{ij}^n = \sum_{i,j\in \ZZ^2} y_j \rho_{ij}^0.
\end{eqnarray*}

$(ii)$ Bound on the second moment: there exists a constant $C>0$ such that for all $n\in \NN^*$, we have
\begin{equation}\label{boundmoment2}
M_2^n := \sum_{i,j\in \ZZ^2} (x_i^2+y_j^2)\rho_{ij}^n \Delta x
\Delta y \leq e^{C t_n} \big(M_2^0 + 1\big) -1,
\end{equation}
where we recall that $t_n=n\Delta t$.
\end{lemma}

\begin{proof}
We first notice that due to Lemma \ref{lem:CFL}, we have that for all $n,i,j$ the sequence $(\rho_{ij}^n)$
is nonnegative.

$(i)$ The mass conservation is directly obtained by summing over $i$ and $j$ equation \eqref{dis_num}.
For the center of mass, we have from \eqref{dis_num} after using a discrete integration by parts~:
$$
\begin{array}{ll}
\ds \sum_{i,j\in \ZZ^2} x_i\rho_{ij}^{n+1} = & \ds \sum_{i,j\in
\ZZ^2} x_i\rho_{ij}^n - \frac{\Delta t}{\Delta x} \sum_{i,j\in
\ZZ^2} {a_x}^n_{i+1/2j} \, \rho_{i+1/2 j}^n \big(x_i-x_{i+1}\big)
  \\[5mm]
& \ds + \frac{\Delta t}{2\Delta x} w_\infty  \sum_{i,j\in \ZZ^2}
\rho_{ij}^n \big(x_{i-1}-2x_i+x_{i+1}\big).
\end{array}
$$
From the definition $x_i=i\Delta x$, we deduce
$$
\sum_{i,j\in \ZZ^2} x_i\rho_{ij}^{n+1} = \sum_{i,j\in \ZZ^2}
x_i\rho_{ij}^n - \Delta t \sum_{i,j\in \ZZ^2} {a_x}^n_{i+1/2j} \,
\rho_{i+1/2 j}^n.
$$
By definition of the macroscopic velocity \eqref{ai12jbis}, we have
$$
\sum_{i,j\in \ZZ^2} {a_x}^n_{i+1/2j} \, \rho_{i+1/2 j}^n =
\frac{1}{\Delta x\Delta y} \sum_{i,j} \sum_{k,\ell}
D_xW_{ij}^{k\ell}\, \rho_{k+1/2\ell}^n \, \rho_{i+1/2 j}^n.
$$
Since the function $\pa_xW$ is odd, we deduce that $D_xW_{ij}^{k\ell}=-D_xW^{ij}_{k\ell}$.
Then by exchanging the role of $i,j$ and $k,\ell$ is the latter sum, we deduce that
it vanishes. Thus,
$$
\sum_{i,j\in \ZZ^2} x_i\rho_{ij}^{n+1} = \sum_{i,j\in \ZZ^2} x_i\rho_{ij}^n
$$
and we proceed in the same way with $y_j$ instead of $x_i$.

$(ii)$ For the second moment, still using \eqref{dis_num} and a discrete integration by parts,
we get
$$
\begin{array}{ll}
\ds \sum_{i,j\in \ZZ^2} x_i^2 \rho_{ij}^{n+1} = & \ds \sum_{i,j\in
\ZZ^2} x_i^2\rho_{ij}^n - \frac{\Delta t}{\Delta x} \sum_{i,j\in
\ZZ^2} {a_x}^n_{i+1/2j} \, \rho_{i+1/2 j}^n
\big(x_i^2-x_{i+1}^2\big)
  \\[5mm]
& \ds + \frac{\Delta t}{2\Delta x} w_\infty  \sum_{i,j\in \ZZ^2}
\rho_{ij}^n \big(x_{i-1}^2-2x_i^2+x_{i+1}^2\big).
\end{array}
$$
By definition $x_i=i\Delta x$, we have
$(x_i^2-x_{i+1}^2)=-2x_{i+1/2}\, \Delta x$ and
$(x_{i-1}^2-2x_i^2+x_{i+1}^2)=2 \Delta x^2$. Thus,
$$
\sum_{i,j\in \ZZ^2} x_i^2 \rho_{ij}^{n+1} = \sum_{i,j\in \ZZ^2}
x_i^2 \rho_{ij}^{n} + 2\Delta t \sum_{i,j\in \ZZ^2}
{a_x}^n_{i+1/2j} \, \rho_{i+1/2 j}^n \, x_{i+1/2} + w_\infty
\Delta t \Delta x,
$$
where we have used the conservation of the mass.
From Lemma \ref{lem:CFL}, we deduce that $|{a_x}^n_{i+1/2j}|\leq w_\infty$. Thus, after
applying a Cauchy-Schwarz inequality and using the mass conservation, we get
$$
\Big|\sum_{i,j\in \ZZ^2} {a_x}^n_{i+1/2j} \, \rho_{i+1/2 j}^n \,
x_{i+1/2} \Delta x\Delta y \Big| \leq \frac{w_\infty}{2} \Big( 1 +
\sum_{i,j\in \ZZ^2} x_{i+1/2}^2 \, \rho_{i+1/2 j}^n \Delta x
\Delta y \Big).
$$
We deduce then that there exists a nonnegative constant $C$ such that
$$
\sum_{i,j\in \ZZ^2} x_i^2 \rho_{ij}^{n+1}\Delta x\Delta y \leq
\Big(1+C\Delta t\Big)\sum_{i,j\in \ZZ^2} x_i^2 \rho_{ij}^{n}\Delta
x\Delta y + C\Delta t.
$$
Doing the same with the term $\sum_{i,j\in \ZZ^2} y_j^2
\rho_{ij}^{n+1}$, we deduce that there exists a nonnegative
constant $C$ such that 
$$
M_2^{n+1} \leq \big(1+C\Delta t\big) M_2^n + C\Delta t. 
$$
We conclude the proof using a discrete Gronwall Lemma.
\end{proof}

%%%%%%%%%%%%%%%%%%%%%%%%%%%%%%%%%%%%%%%%%%%%%%%%%%%%%%%%%%%%%%%%%%%%%%%%%
\subsection{Convergence of the numerical approximation}

Let us denote by $\Delta = \max\{\Delta x,\Delta y\}$. We define the
reconstruction 
\beq\label{rhodelta} 
\rho_\Delta(t,x,y) =
\sum_{n\in\NN} \sum_{i\in \ZZ}\sum_{j\in \ZZ} \rho_{ij}^n {\bf
1}_{[n\Delta t,(n+1)\Delta t)\times C_{ij}}(t,x,y), 
\eeq
Therefore, we have by definition of $a_{ij}^n=({a_x}_{ij}^n,{a_y}_{ij}^n)$ in \eqref{def:ai12j} that
$$
a_{ij}^n = \frac{1}{\Delta x\Delta y} \iint_{C_{ij}} \nabWchapo *
\rho_\Delta(t_n,x,y)\,dxdy.
$$
In the same manner, we define
$$
a_\Delta(t,x,y) = \sum_{n\in\NN} \sum_{i\in\ZZ}\sum_{j\in \ZZ}
a_{ij}^n {\bf 1}_{[n\Delta t,(n+1)\Delta t)\times C_{ij}}(t,x,y).
$$

Then we have the following convergence result:
\begin{theorem}\label{convnum}
Let us assume that $W$ satisfies {\bf (A0)-(A2)} and consider
$\rho^{ini}\in \calP_2(\RR^2)$. We define $\rho_{ij}^0$ by
\eqref{disrho0}. Let $T>0$ be fixed. Then, if (\ref{CFL}) is
satisfied, the discretization $\rho_\Delta$ converges weakly in
$\calM_b([0,T]\times\RR^2)$ towards the solution $\rho$ of
Theorem \ref{Exist} as $\Delta:=\max\{\Delta x,\Delta y\}$ goes to
$0$ with $\Delta t$ satisfying \eqref{CFL}.
\end{theorem}

\begin{proof}
From Lemma \ref{lem:CFL}, we have that $\rho_{ij}^n\geq 0$
provided the condition \eqref{CFL} is satisfied. Moreover, by
conservation of the mass we deduce that the sequence nonnegative
bounded measures $(\rho_\Delta)_\Delta$ satisfies for all $t\in
[0,T]$, $|\rho_\Delta(t)|(\RR^2)=1$. Therefore, we can extract a
subsequence, still denoted $(\rho_\Delta)_\Delta$, converging for the weak
topology towards $\rho$ as $\Delta t$, $\Delta x$ and $\Delta y$
go to $0$ satisfying \eqref{CFL}, i.e. $\forall \, \phi \in
C_0([0,T]\times\RR^2)$,
$$
\int_0^T\iint_{\RR^2} \phi(t,x,y) \rho_\Delta(t,x,y) \,dxdydt
\longrightarrow \int_0^T\iint_{\RR^2} \phi(t,x,y) \rho(t,dx,dy)
\,dt.
$$
Actually, due to the estimate \eqref{boundmoment2} in Lemma \ref{bounddismom}, we can deduce that
$$
\int_0^T\iint_{\RR^2} (x^2+y^2) \rho(t,dx,dy)
\,dt.
$$

We choose $\Delta t>0$ and $N_T\in \NN^*$ such that condition
\eqref{CFL} holds and $T=\Delta t N_T$. Let $\phi\in
\calD([0,T]\times \RR^2)$ be smooth and compactly supported. We
denote
$$
\psi^n_{i,j} = \int_{t_n}^{t_{n+1}}\iint_{C_{ij}} \phi(t,x,y)\,dtdxdy,
$$
such that
$$
\int_0^T \iint_{\RR^2} \rho_\Delta(t,x,y)\phi(t,x,y)\,dtdxdy =
\sum_{n=0}^{N_T}\sum_{i\in \ZZ}\sum_{j\in \ZZ} \rho_{ij}^n
\psi^n_{i,j}.
$$
In particular, we have
$$
\begin{array}{r}
\ds \sum_{n,i,j} \frac{1}{\Delta t} \big(\rho_\Delta
(t_{n+1},x_i,y_j)-\rho_\Delta(t_n,x_i,y_j)\big) \psi^n_{i,j}
\ds =- \sum_{n,i,j} \rho_{i,j}^n \frac{\psi^n_{i,j}-\psi^{n-1}_{i,j}}{\Delta t}  \\[2mm]
\ds = - \int_0^T \iint_{\RR^2}
\rho_\Delta(t,x,y)\frac{\phi(t,x,y)-\phi(t-\Delta t,x,y)}{\Delta
t}\,dtdxdy.
\end{array}
$$
We have $\phi(t,x,y)-\phi(t-\Delta t,x,y)=\pa_t\phi(t,x,y) \Delta
t+O(\Delta t^2)$. From the weak convergence of $\rho_\Delta$ and
the fact that $\rho_\Delta$ is a bounded measure with a bound not
depending on the mesh, we deduce that the latter integral
converges to
$$
- \int_0^T \iint_{\RR^2} \pa_t\phi(t,x,y)\rho(t,dx,dy)\,dt.
$$
By the same token, we have
$$
\begin{array}{l}
\ds \sum_{n,i,j} \frac{1}{2\Delta x}\big(\rho_\Delta (t_n,x_{i+1},y_j)-2\rho_\Delta (t_{n},x_i,y_j)+\rho_\Delta (t_{n},x_{i-1},y_j)\big) \psi^n_{i,j} \\[2mm]
\ds = \int_0^T \iint_{\RR^2}
\rho_\Delta(t,x,y)\frac{\phi(t,x+\Delta
x,y)-2\phi(t,x,y)-\phi(t,x-\Delta x,y)}{2\Delta x}\,dtdxdy,
\end{array}
$$
Using the fact that $|\phi(t,x+\Delta
x,y)-2\phi(t,x,y)-\phi(t,x-\Delta x,y)|\leq
\|\pa_{xx}\phi\|_\infty \Delta x^2$, we deduce that this latter
integral converges towards $0$ as $\Delta t$, $\Delta x$ and
$\Delta y$ go to $0$. Futhermore, we have \beq\label{termflux}
\begin{array}{l}
\ds \sum_{n,i,j} \frac{1}{\Delta x}\big({a_x}^n_{i+1/2j}\rho_{i+1/2j}^n- {a_x}^n_{i-1/2j}\rho_{i-1/2j}^n\big) \psi^n_{i,j} =  \\[2mm]
\ds = -\frac{1}{4\Delta x}\sum_{n,i,j}
{a_x}_{ij}^n\rho_{ij}^n\big(\psi^n_{i+1,j}-\psi^n_{i-1,j}\big) +
{a_x}_{i+1j}^n\rho_{ij}^n \big(\psi^n_{i+1,j}-\psi^n_{i,j}\big) +
{a_x}_{i-1j}^n\rho_{ij}^n \big(\psi^n_{i,j}-\psi^n_{i-1,j}\big)
 \\[3mm]
\ds = -\frac{1}{4\Delta x} \int_0^T\!\! \iint_{\RR^2} \!
\Big({a_x}_\Delta(t,x,y)\rho_\Delta(t,x,y) \big(\phi(t,x+\Delta x,y)-\phi(t,x-\Delta x,y)\big) \\[1mm]
\ds \qquad\qquad\qquad\quad +{a_x}_\Delta(t,x+\Delta x,y)\rho_\Delta(t,x,y) \big(\phi(t,x+\Delta x,y)-\phi(t,x,y)\big) \\[1mm]
\ds \qquad\qquad\qquad\quad +{a_x}_\Delta(t,x-\Delta
x,y)\rho_\Delta(t,x,y) \big(\phi(t,x,y)-\phi(t,x-\Delta
x,y)\big)\Big)\,dtdxdy.
\end{array}
\eeq
Using a Taylor expansion, the mass conservation and the bound \eqref{borngradW},
we deduce from \eqref{termflux} that
\beq\label{termflux2}
\begin{array}{c}
\ds \sum_{n,i,j} \frac{1}{\Delta
x}\big({a_x}^n_{i+1/2j}\rho_{i+1/2j}^n -
{a_x}^n_{i-1/2j}\rho_{i-1/2j}^n\big) \psi^n_{i,j}
 = -\frac{1}{4} \int_0^T\!\! \iint_{\RR^2} \!
\Big(2{a_x}_\Delta(t,x,y)\rho_\Delta(t,x,y) \,\pa_x\phi(t,x,y) \\[1mm]
\ds \qquad\qquad +\big({a_x}_\Delta(t,x+\Delta
x,y)+{a_x}_\Delta(t,x-\Delta x,y)\big) \rho_\Delta(t,x,y)
\pa_x\phi(t,x,y)\Big)\,dtdxdy + O(\Delta x).
\end{array}
\eeq
Then, from \eqref{def:ai12j}, we deduce that for any test function $\xi$ we have on the one hand
	$$\begin{array}{l}
\ds \iint_{\RR^2} {a_x}_\Delta \rho_\Delta \xi (x_1,y_1) \,dx_1dy_1 = \\[3mm]
\qquad \ds =\sum_{i,j,k,\ell} \frac{1}{\Delta x\Delta
y}\iint_{C_{ij}}\iint_{C_{k\ell}} \rho_{k\ell}^n\,\rho_{ij}^n
\widehat{\pa_xW}(x-x',y-y')\,dxdydx'dy' \iint_{C_{ij}}
\xi(x_1,y_1) \,dx_1dy_1,
	\end{array}$$
on the other hand,
	$$
\iint_{\RR^2} \widehat{\pa_xW}*\rho_\Delta \, \rho_\Delta
\xi(x,y)\,dxdy = \sum_{i,j,k,\ell} \iint_{C_{ij}}\iint_{C_{k\ell}}
\rho_{k\ell}^n\,\rho_{ij}^n\,
\widehat{\pa_xW}(x-x',y-y')\xi(x,y)\,dx'dy'dxdy.
	$$
Moreover, for any test function $\xi$ smooth and compactly supported, we have for all $x,y \in C_{ij}$,
$$
\frac{1}{\Delta x\Delta y} \iint_{C_{ij}} \xi(x_1,y_1) \,dx_1dy_1
= \xi(x,y) + O(\Delta x) + O(\Delta y).
$$
Thus we have
$$
\iint_{\RR^2} {a_x}_\Delta \rho_\Delta \xi (x,y) \,dxdy =
\iint_{\RR^2} \widehat{\pa_xW}*\rho_\Delta \, \rho_\Delta
\xi(x,y)\,dxdy + O(\Delta x) + O(\Delta y).
$$

Finally, we deduce from \eqref{termflux2} \beq\label{termflux3}
\sum_{n,i,j} \frac{1}{\Delta
x}\big({a_x}^n_{i+1/2j}\rho_{i+1/2j}^n -
{a_x}^n_{i-1/2j}\rho_{i-1/2j}^n\big) \psi^n_{i,j} =
 -\frac 12 \big( I_1 + I_2 \big) + O(\Delta x)+O(\Delta y),
\eeq
where
$$
I_1 = \int_0^T\!\! \iiiint_{\RR^4} \widehat{\pa_xW} (x-x',y-y')
\rho_\Delta(t,x',y')\rho_\Delta(t,x,y) \pa_x\phi(t,x,y) \,dxdy dt,
$$
$$
\begin{array}{ll}
\ds I_2 = \frac 12 \int_0^T\!\! \iiiint_{\RR^4} &
\ds \!\!\! \big(\widehat{\pa_xW} (x+\Delta x-x',y-y')+ \widehat{\pa_xW} (x-\Delta x-x',y-y')\big)  \\
&\ds \rho_\Delta(t,x',y') \rho_\Delta(t,x,y)\pa_x\phi(t,x,y)
\,dx'dy'dxdydt.
\end{array}
$$
As a direct consequence of Lemma \ref{lemstab_a}, we have
$$
I_1 \underset{\Delta \to 0}{\longrightarrow}
\int_0^T\!\!\iint_{\RR^2} \widehat{\pa_xW}\!*\!\rho
(t,x,y)\rho(t,x,y) \pa_x\phi(t,x,y) \,dtdxdy.
$$
For the term $I_2$, we proceed as in the proof of Lemma \ref{lemstab_a}. We recall the main ingredients of this proof.
First, using the symmetry of $W$, we write
$$
\begin{array}{ll}
\ds I_2 = \frac 14 \int_0^T\!\! \iiiint_{\RR^4} &
\ds \!\!\! \big(\widehat{\pa_xW} (x+\Delta x-x',y-y')+ \widehat{\pa_xW} (x-\Delta x-x',y-y')\big)  \\
&\ds \rho_\Delta(t,x',y') \rho_\Delta(t,x,y)
\big(\pa_x\phi(t,x,y)-\pa_x\phi(t,x',y')\big) \,dx'dy'dxdydt.
\end{array}
$$
We introduce the set $D_\alpha = \{(x,y,x',y')$ s.t.
$|x-x'|+|y-y'|<\alpha\}$ for some positive coefficient
$\alpha<\Delta x$ and split the latter integral into the sum of
the integral over $\RR^4\setminus D_\alpha$ and over $D_\alpha$.
Using the uniform continuity of $\pa_x\phi$, we deduce that the
integral over $D_\alpha$ is small for small $\alpha$. Then, using
the fact that by continuity of $\pa_xW$ on $\RR^4\setminus \{0\}$,
we have for all $(x,x',y,y')\in\RR^4\setminus D_\alpha$
$$
\lim_{\Delta x \to 0} \big(\widehat{\pa_xW} (x+\Delta x-x',y-y')+
\widehat{\pa_xW} (x-\Delta x-x',y-y')\big) = \widehat{\pa_xW}
(x-x',y-y'),
$$
we deduce
$$
I_2 \underset{\Delta \to 0}{\longrightarrow}
\int_0^T\!\!\iint_{\RR^2} \widehat{\pa_xW}\!*\!\rho
(t,x,y)\rho(t,x,y) \pa_x\phi(t,x,y) \,dtdxdy.
$$
Therefore, we conclude from \eqref{termflux3}
$$
\begin{array}{l}
\ds \lim_{\Delta \to 0} \sum_{n,i,j} \frac{1}{\Delta x}\big({a_x}^n_{i+1/2j}\rho_{i+1/2j}^n- {a_x}^n_{i-1/2j}\rho_{i-1/2j}^n\big) \psi^n_{i,j} = \\[2mm]
\qquad\qquad\qquad \ds
= -\int_0^T\!\!\iint_{\RR^2} \pa_x\phi(t,x,y) \widehat{\pa_x W}\!*\!\rho(t,x,y) \rho(t,x,y)\,dtdxdy.
\end{array}
$$

Finally, multiplying equation \eqref{dis_num} by $\psi^n_{i,j}$,
summing over $n$, $i$, $j$ and taking the limit $\Delta t$,
$\Delta x$, $\Delta y$ to $0$, we obtain
$$
\int_0^T\iint_{\RR^2} \big(\pa_t\phi(t,x,y) + \nabWchapo\!*\!\rho(t,x,y)\cdot\nabla\phi(t,x,y)\big)\rho(t,dx,dy) = 0.
$$
Thus $\rho$ is a solution in the sense of distributions of the aggregation equation \eqref{EqInter}. We proceed now as in the proof of Theorem \ref{Exist}. Due to the assumptions on the potential, we have that $-\nabWchapo\!*\!\rho \in L^2((0,T),L^2(\rho(t)))$. Then, we deduce that $\rho\in AC^2([0,T],\calP_2(\RR^2))$ using \cite[Theorem 8.3.1]{Ambrosio}.
By uniqueness of the gradient flow solution and the equivalence Theorem \ref{equivalence}, we conclude that $\rho$ is the solution of Theorem \ref{Exist}.
Since the limit is unique, we deduce that the whole sequence is converging towards this limit.
\end{proof}

%%%%%%%%%%%%%%%%%%%%%%%%%%%%%%%%%%%%%%%%%%%%%%%%%%%%%%%%%%%%%%%%%%%%%%%%%
\subsection{Numerical simulations}

We perform in this subsection some numerical simulations obtained by implementing the scheme
described above \eqref{dis_num}--\eqref{def:ai12j}. We will consider two examples of potential which fit the assumptions {\bf (A0)--(A2)}, that is
$$
W_1(x) = 1-e^{-5|x|}\ ; \qquad W_2(x) = |x|.
$$
For such potentials it is known, see \cite[Section 4]{Carrillo}, that finite
time collapse occurs. More precisely, for any compactly suported initial data,
there exists a finite time beyond which the solution is given by a single
Dirac Delta mass located at the center of mass. We verify here that we can observe such
phenomena thanks to the numerical scheme introduced above.

In our numerical simulations, we consider an initial data given by the sum of three regular bumps:
$$
\begin{array}{ll}
\ds \rho^0(x,y) = &\ds \exp(-C_x (x-1/4)^2-C_x(y-1/3)^2)+ \exp(-C_x(x-0.8)^2-C_x(y-0.6)^2) \\[2mm]
&\ds + 0.9 \exp(-C_x (x-0.4)^2-C_x (y-0.6)^2),
\end{array}
$$
with $c_x=100$.

\begin{figure}[ht!]
\begin{center}
  \includegraphics[width=7cm]{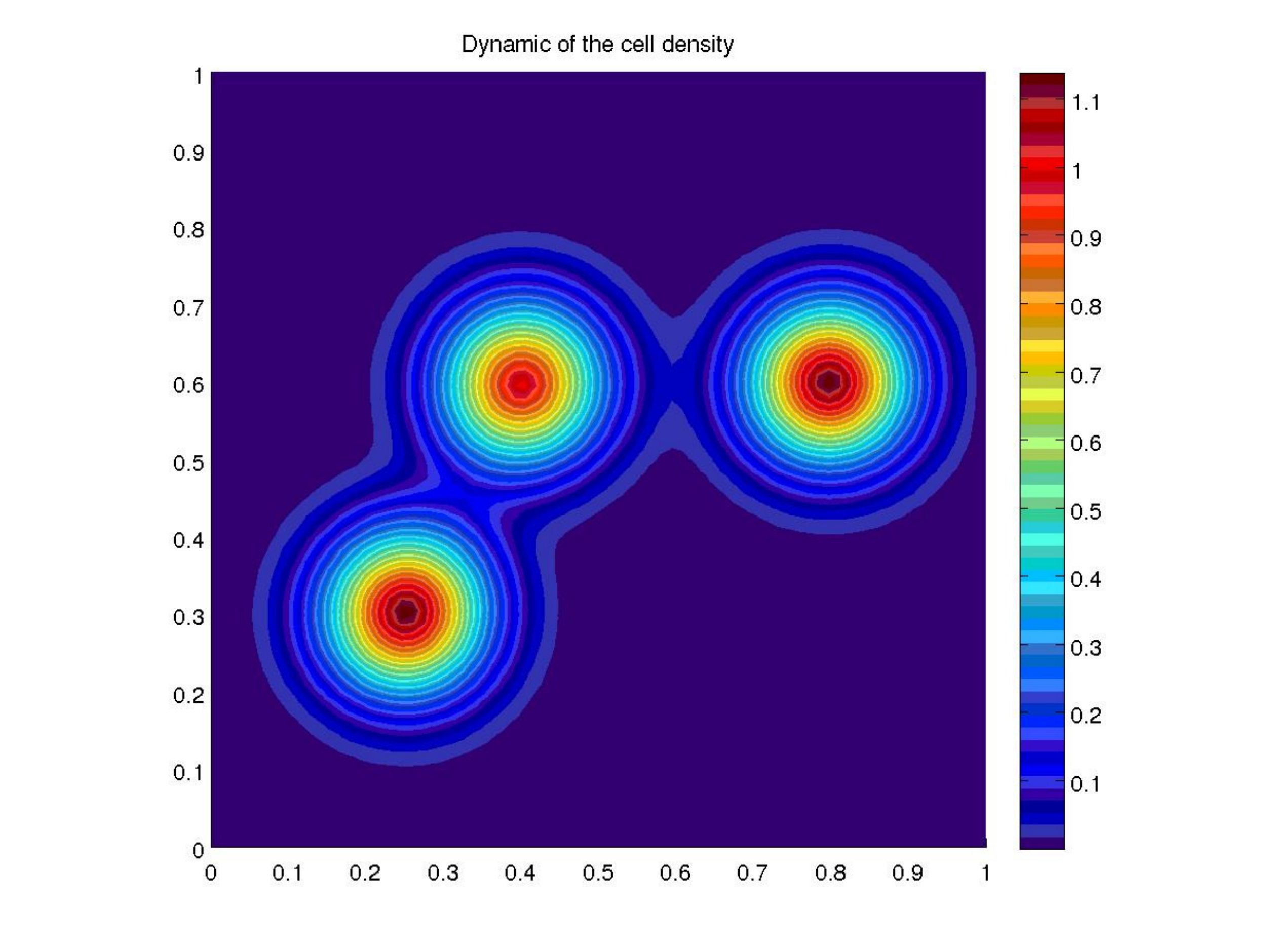}
  \includegraphics[width=7cm]{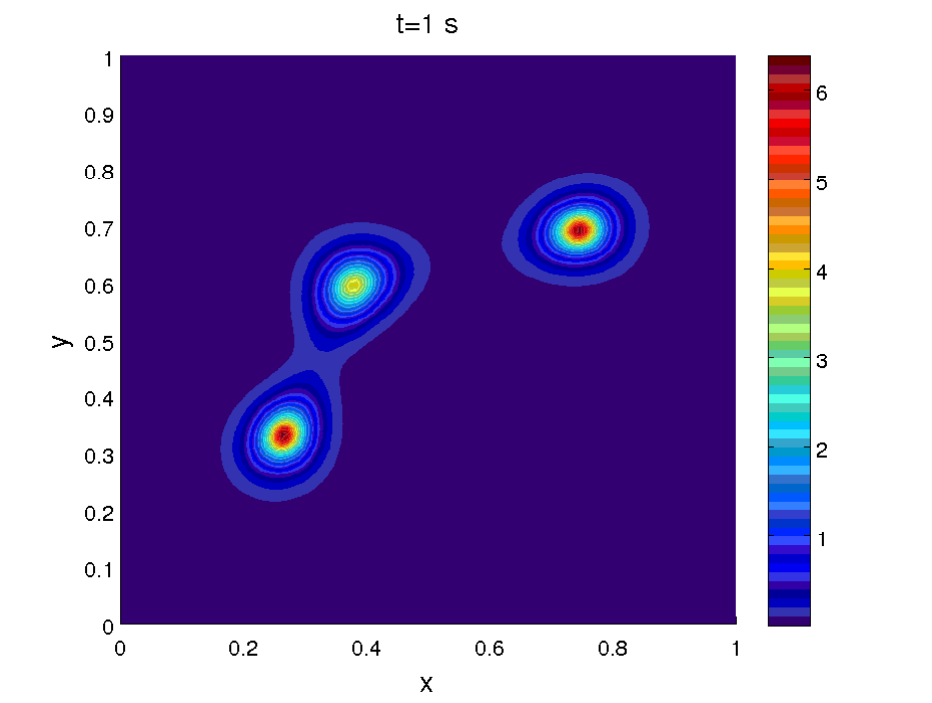} \\
  \includegraphics[width=7cm]{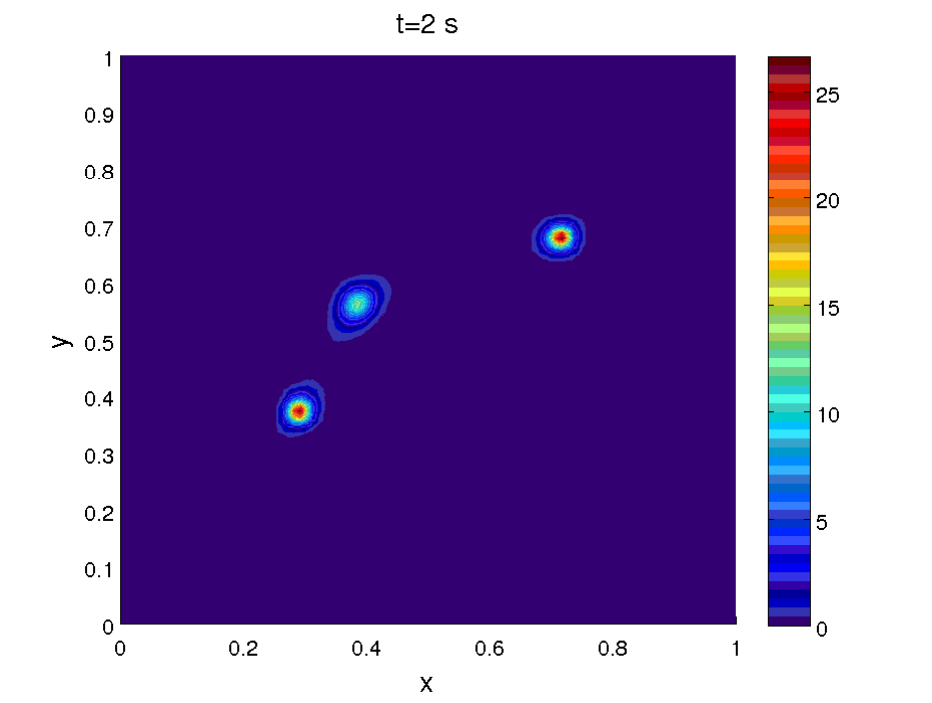}
  \includegraphics[width=7cm]{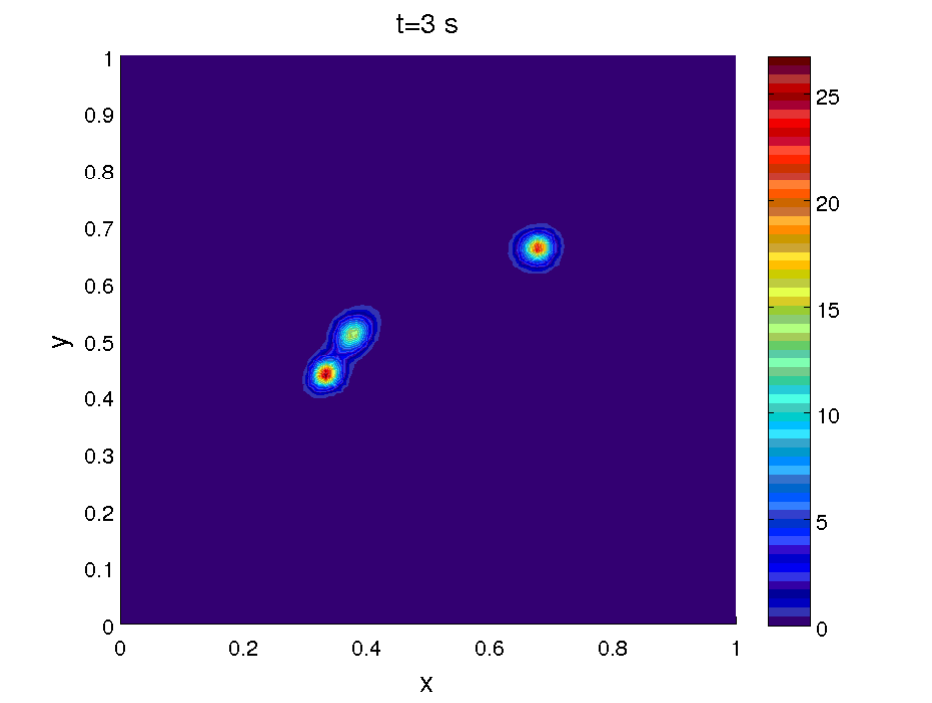} \\
  \includegraphics[width=7cm]{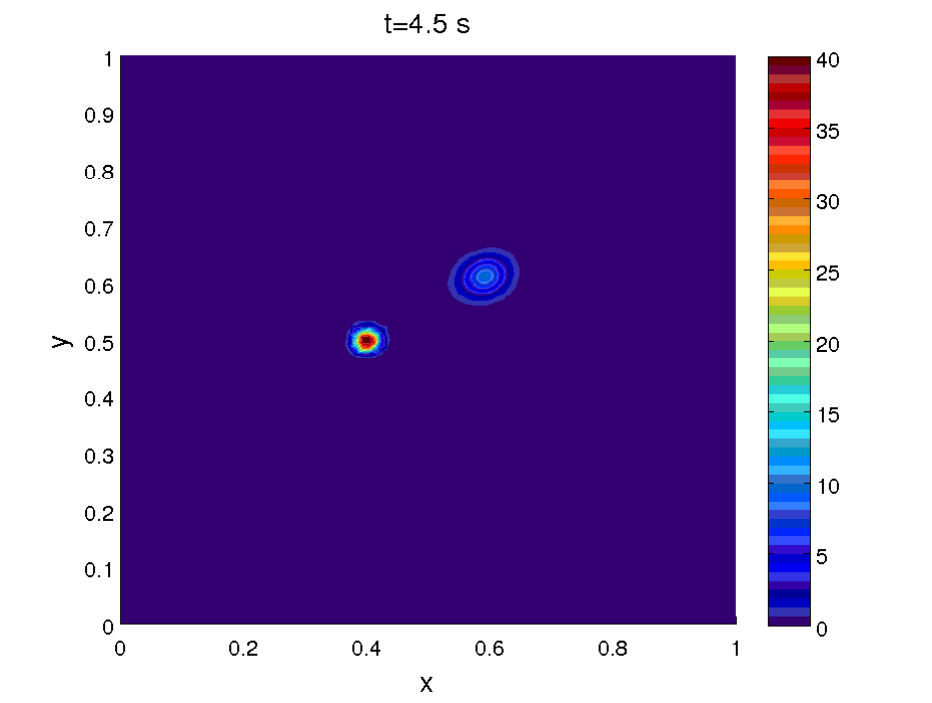}
  \includegraphics[width=7cm]{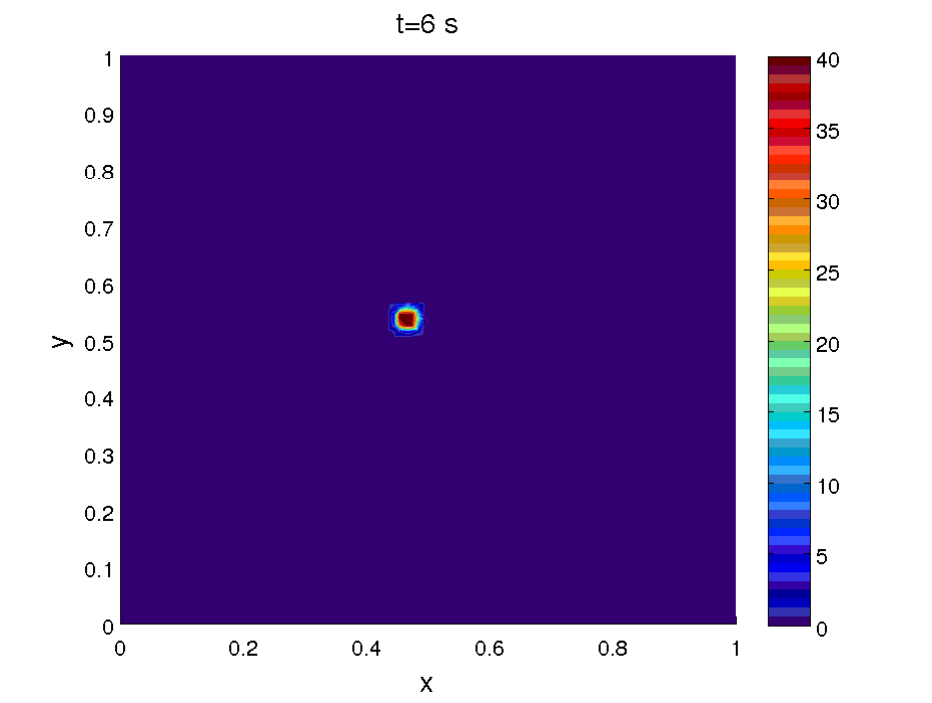}
  \caption{Dynamics of the cell density $\rho$ with intial data given by the sum of three bumps
  in the case $W_1(x)=1-e^{-5|x|}$.}\label{fig1}
\end{center}
\end{figure}

\begin{figure}[ht!]
\begin{center}
  \includegraphics[width=7cm]{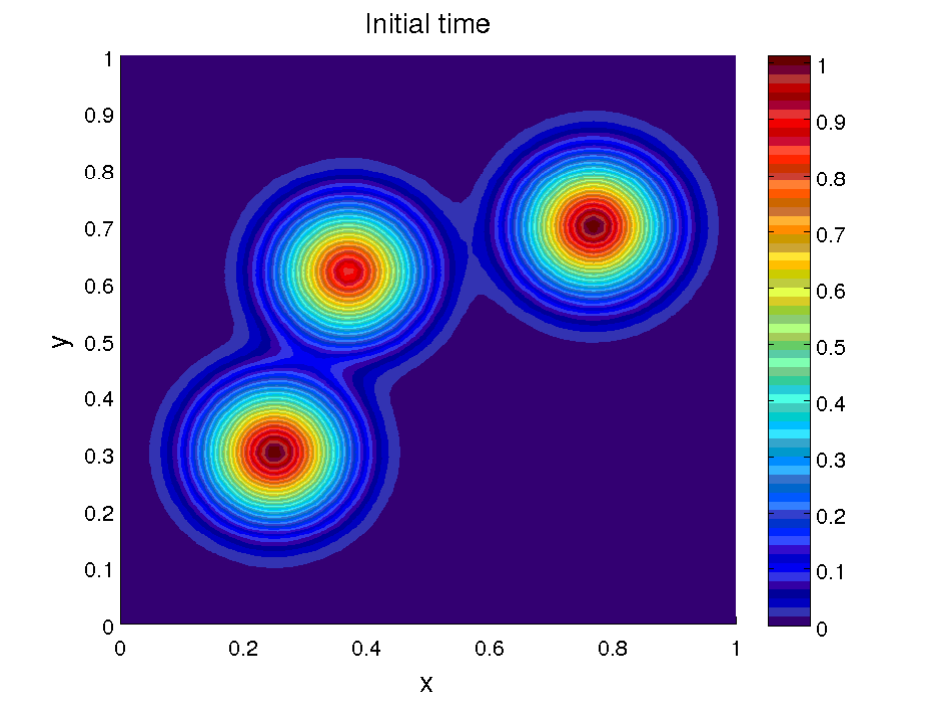}
  \includegraphics[width=7cm]{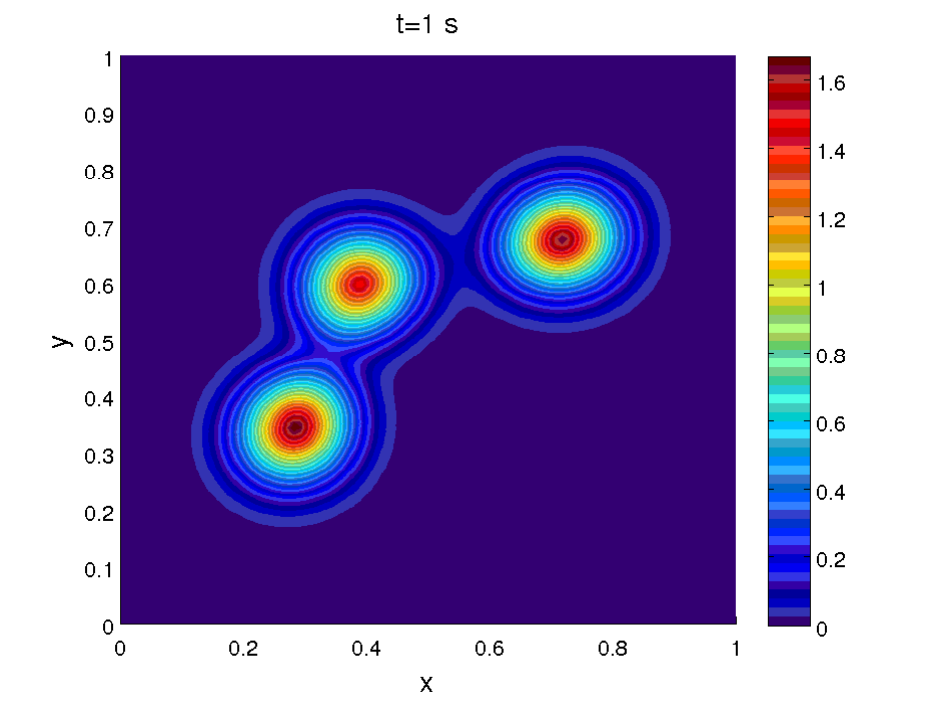} \\
  \includegraphics[width=7cm]{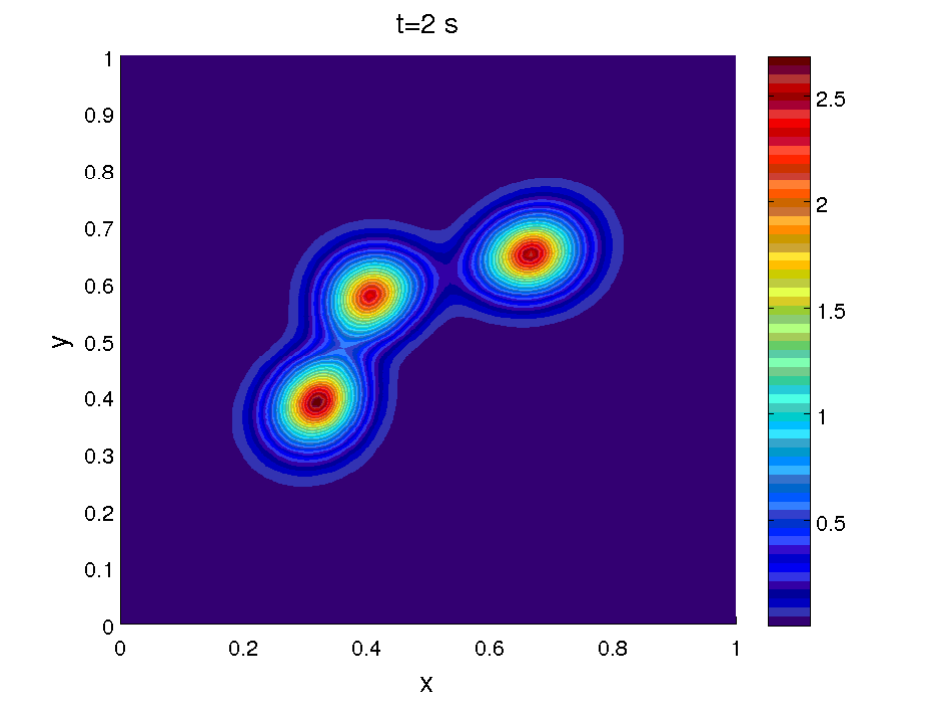}
  \includegraphics[width=7cm]{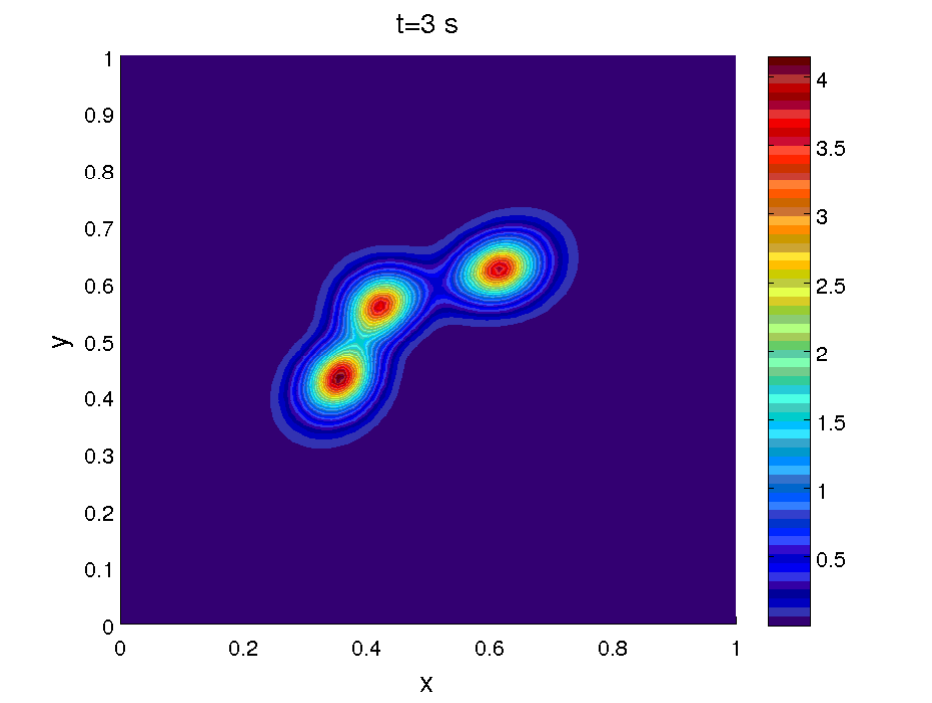} \\
  \includegraphics[width=7cm]{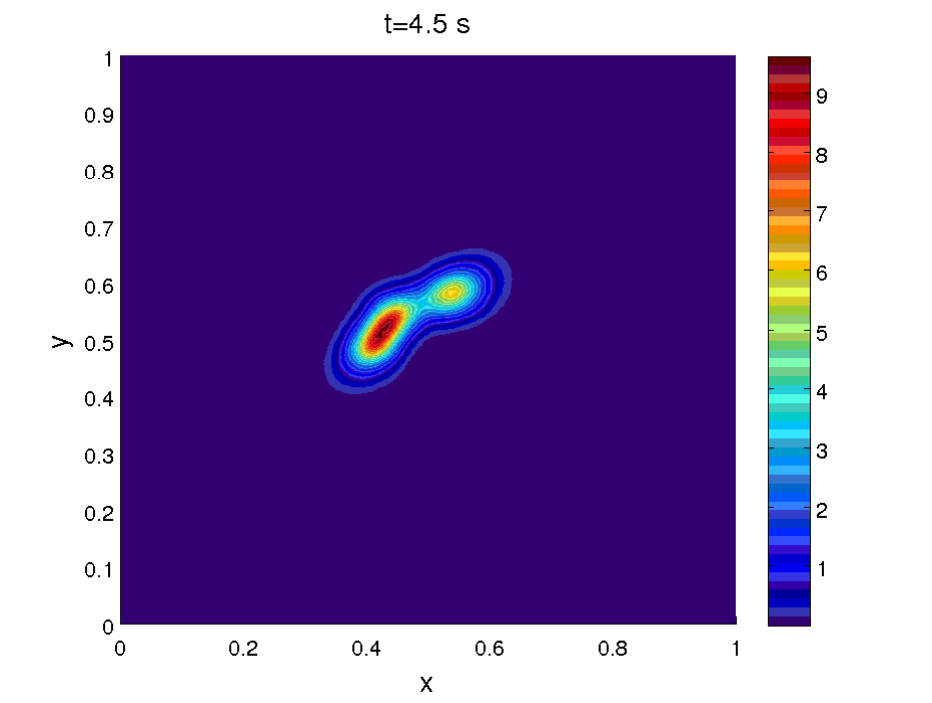}
  \includegraphics[width=7cm]{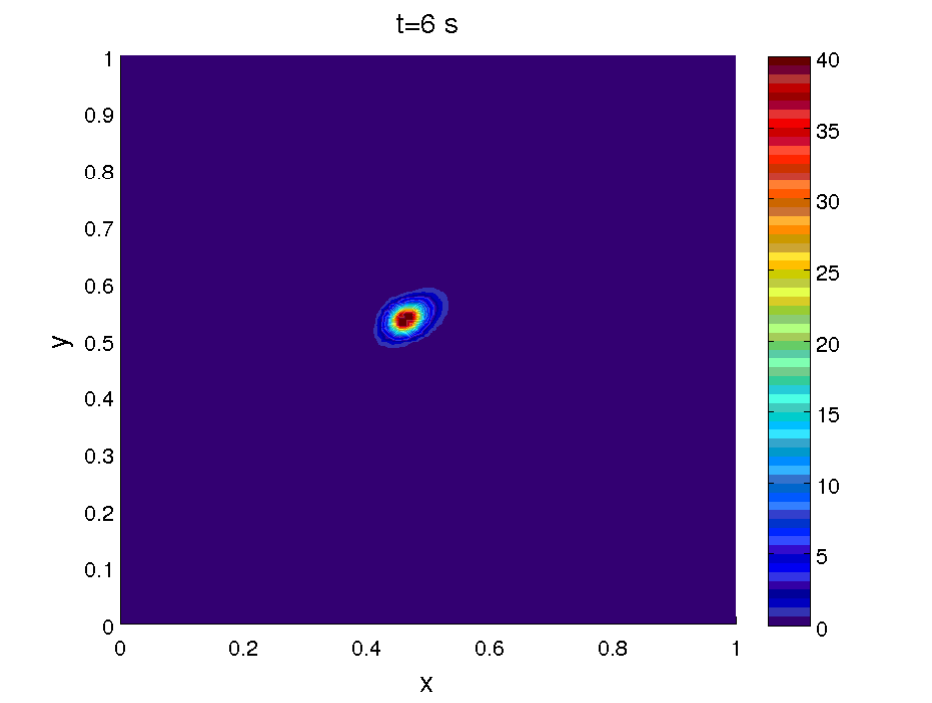}
  \caption{Dynamics of the cell density $\rho$ with intial data given by the sum of three bumps
  in the case $W_2(x)=|x|$.}\label{fig2}
\end{center}
\end{figure}

Due to the finite time collapse result, we expect the convergence in finite time of the solution
towards a single Dirac Delta. In fact, this is what we observe in Figure \ref{fig1} for $W_1$ and in Figure \ref{fig2} for $W_2$. However, comparing the two Figures, the qualitative properties of the convergence towards a single Dirac Delta are not the same depending on the choice of the potential.

In fact, within the dynamics given in Figure \ref{fig1}, we can distinguish two phases in the simulation.
In a first phase, we notice the concentration of the density into small masses : we can
consider that the numerical solution for time $t=1.8\ s$ is a sum of three numerical Dirac masses with
small numerical diffusion. Then these three masses aggregate into two and finally one single mass.
On the contrary, for the potential $W_2$, we observe in Figure \ref{fig2} that the numerical solution stays
regular and bounded until it forms one single bump and then it collapses. 

This tends to indicate the existence of two different time scales: the one corresponding to a radial self-similar collapse onto a single Dirac, and the one corresponding to the interactions between different Dirac Deltas. In the case of the potential $W_1$, we observe a faster time scale for the self-similar blow-up of regular solutions into several Dirac Deltas, then the trajectories are given by the sticky particle dynamics for these aggregates. Whereas for the potential $W_2$ the time scale of the self-similar blow-up is slower compared to the dynamics of the attraction of the aggregates, and then the blow up occurs after
all regular bumps aggregate into a single regular bump before the final fate of total collapse.

A very nice feature of this numerical scheme is that it allows for simulations after the first blow-up happens with seemingly good approximation in the measure sense by comparison to the particle simulations, see the one dimensional case \cite{sinum}. The regularization induced on the Dirac Deltas by the numerical diffusion of the scheme does not seem to change the qualitative properties of the solution.

\newpage

\appendix

\Section*{Technical Lemmas}

In this appendix we state some technical lemmas which are used in the paper.
\begin{lemma}\label{lemA}
Let us assume that $W$ satisfies assumptions {\bf (A0)--(A2)}.
Let $(\rho_n)_{n\in \NN}$ be a sequence of measures in $\calP_2(\RR^d)$ such that
$\rho_n \rightharpoonup \rho$ weakly as measures. Then
$$
\lim_{n\to +\infty} \int_{x\neq y} \nabla W(x-y) \rho_n(dy) =
\int_{x\neq y} \nabla W(x-y) \rho(dy), \quad \mbox{ for a.e. } x \in \RR^d.
$$
\end{lemma}
\begin{proof}
We consider a regularization of $W$ by $W_k$ with $k\in\NN$,
$W_k \in C^1(\RR^d)$, $W_k(-x)=W_k(x)$, $|\nabla W_k|\leq |\nabla W|\leq w_\infty$, and
\beq\label{boundWapprox}
\sup_{x\in \RR^d\setminus B(0,\frac1k)} |\nabla W_k(x)-\nabla W(x)|\leq \frac1k.
\eeq
By definition of the weak convergence of measures, we have
\beq\label{ineqA1}
\lim_{n\to +\infty} \int_{x\neq y} \nabla W_k(x-y) \rho_n(dy) =
\int_{x\neq y} \nabla W_k(x-y) \rho(dy), \quad \mbox{ for a.e. } x \in \RR^d.
\eeq
In fact, we can remove the point $y=x$ in the integral since by construction
$\nabla W_k$ is odd, then $\nabla W_k(0)=0$.
Moreover for all $n\in \NN$, we have that
\begin{align}
\Big|\int_{x\neq y} \nabla (W_k-W)(x-y) \rho_n(dy)\Big| \leq &\,
\Big|\int_{B(x,\frac1k)\setminus\{x\}} \nabla (W_k-W)(x-y) \rho_n(dy)\Big|  \nonumber\\[2mm]
& +\Big|\int_{\RR^d \setminus B(x,\frac1k)} \nabla (W_k-W)(x-y) \rho_n(dy)\Big|.\label{jejeje}
\end{align}
Given $\eps >0$, we use the property \eqref{boundWapprox} to get an estimate on the second term in \eqref{jejeje}
\begin{equation}\label{jojo}
\Big|\int_{\RR^d \setminus B(x,\frac1k)} \nabla (W_k-W)(x-y) \rho_n(dy)\Big|\leq \frac1k \leq \eps
\end{equation}
for $k\geq K_1$.

Now, we fix $K_2\geq K_1$ such that
\beq\label{ineq41}
\rho\big(B(x,\frac {2}{K_2})\setminus\{x\}\big) \leq \frac {\eps}{4}.
\eeq
We choose a continuous function $0\leq \xi\leq 1$ such that $\xi(x)=1$ on
$B(x,\frac {1}{K_2})$ and $\xi(x)=0$ on $\RR^d\setminus B(x,\frac {2}{K_2})$.
Then $\xi\in C_c(\RR^d)$ and for all $k\geq K_2$, we have
$$
\begin{array}{ll}
\ds \rho_n\big(B(x,\frac 1k)\setminus\{x\}\big) &\ds \leq
\rho_n\big(B(x,\frac{1}{K_2})\setminus\{x\}\big) \leq
\int_{\RR^d} \xi(x) \rho_n(dx)  \\[3mm]
&\ds \leq \int_{\RR^d} \xi(x)(\rho_n-\rho)(dx)
+\frac \eps 4,
\end{array}
$$
where we use \eqref{ineq41} for the last inequality.
From the weak convergence as measures of $\rho_n$ towards $\rho$, we have that
for $n\geq N_1$ large enough
$$
\Big|\int_{\RR^d} \xi(x)(\rho_n-\rho)(dx)\Big| \leq \frac \eps 4.
$$
Thus, for $k\geq K_2$ we obtain
$$
\rho_n\big(B(x,\frac 1k)\setminus\{x\}\big) \leq \frac \eps 2\,,
$$
uniform in $n\geq N_1$. Therefore, we can bound the first term of the right hand side in \eqref{jejeje} as
$$
\Big|\int_{B(x,\frac1k)\setminus\{x\}} \nabla (W_k-W)(x-y) \rho_n(dy)\Big| \leq 2w_\infty \rho_n(B(x,\frac1k)\setminus\{x\}) \leq w_\infty \eps\,.
$$

Collecting the last inequality with \eqref{jojo}, we deduce that
$$
\lim_{k\to \infty} \int_{x\neq y} \nabla W_k(x-y) \rho_n(dy) =
\int_{x\neq y} \nabla W(x-y) \rho_n(dy),
$$
uniformly for $n\geq N_1$. The same argument shows in particular that
$$
\lim_{k\to \infty} \int_{x\neq y} \nabla W_k(x-y) \rho(dy) =
\int_{x\neq y} \nabla W(x-y) \rho(dy).
$$
We conclude by passing into the limit $k\to \infty$ in \eqref{ineqA1}.
\end{proof}

\begin{lemma}\label{lemB}
Let us assume that $W$ satisfies assumptions {\bf (A0)--(A2)}.
Let $(W_n)_{n\in\NN^*}$ be a sequence of even functions in $C^1(\RR^d)$ satisfying {\bf (A1)}
and \eqref{borngradW} with constants $\lambda$ and $w_\infty$ not depending on $n$ and such that
\beq\label{boundnabWnB}
sup_{x\in \RR^d\setminus B(0,\frac 1n)} \big|\nabla W_n(x)-\nabla W(x)\big| \leq \frac 1n, \qquad \mbox{ for all } n\in \NN^*.
\eeq
Let $(\rho_n)_{n\in \NN}$ be a sequence of measures in $\calP_2(\RR^d)$ such that
$\rho_n \rightharpoonup \rho$ tightly. Then we have
$$\displaystyle \lim_{n\to +\infty} \int_{\RR^d} \nabla W_n(x-y) \rho_n(dy) =
\int_{x\neq y} \nabla W(x-y) \rho(dy), \quad \mbox{ for a.e. } x \in \RR^d.$$
\end{lemma}
\begin{proof}
Let us denote by
$$
a_n(x) := -\int_{\RR^d} \nabla W_n(x-y) \rho_n(dy), \quad \mbox{ and }
a(x) := -\int_{x\neq y} \nabla W(x-y) \rho(dy).
$$
We notice that since $W_n$ is even, we have $\nabla W_n(0)=0$, then
$$
a_n(x) := -\int_{x\neq y} \nabla W_n(x-y) \rho_n(dy).
$$
Let $\eps>0$, from Lemma \ref{lemA} we deduce that there exists $N_1\in \NN^*$
such that for all $n\geq N_1$,
\beq\label{ineq1}
\Big|\int_{x\neq y} \nabla W(x-y) (\rho_n-\rho)(dy)\Big| \leq \frac{\eps}{4}.
\eeq
Then using \eqref{boundnabWnB}, we deduce that
\beq\label{ineq2}
\Big|\int_{x\neq y} \big(\nabla W_n(x-y)-\nabla W(x-y)\big) \rho_n(dy)\Big| \leq
\frac 1n +
\int_{B(x,\frac 1n)\setminus\{x\}} \big|\nabla W_n(x-y)-\nabla W(x-y)\big|
 \rho_n(dy).
\eeq
Now, we proceed as in the proof of Lemma \ref{lemA}.
From assumptions on $W_n$ and $W$, we deduce (see \eqref{borngradW}) that
there exists a constant $C$ such that
\beq\label{ineq3}
\int_{B(x,\frac 1n)\setminus\{x\}} \big|\nabla W_n(x-y)-\nabla W(x-y)\big|
 \rho_n(dy) \leq C \rho_n\big(B(x,\frac 1n)\setminus\{x\}\big)
\eeq
We fix $N_2\geq N_1$ such that
\beq\label{ineq4}
\rho\big(B(x,\frac {2}{N_2})\setminus\{x\}\big) \leq \frac {\eps}{4}.
\eeq
We choose a continuous function $0\leq \xi\leq 1$ such that $\xi(x)=1$ on
$B(x,\frac {1}{N_2})$ and $\xi(x)=0$ on $\RR^d\setminus B(x,\frac {2}{N_2})$.
Then $\xi\in C_c(\RR^d)$ and for all $n\geq N_2$, we have
$$
\begin{array}{ll}
\ds \rho_n\big(B(x,\frac 1n)\setminus\{x\}\big) &\ds \leq
\rho_n\big(B(x,\frac{1}{N_2})\setminus\{x\}\big) \leq
\int_{\RR^d} \xi(x) \rho_n(dx)  \\[3mm]
&\ds \leq \int_{\RR^d} \xi(x)(\rho_n-\rho)(dx)
+\frac \eps 4,
\end{array}
$$
where we use \eqref{ineq4} for the last inequality.
From the tight convergence of $\rho_n$ towards $\rho$, we have that
for $n\geq N_3$ large enough (and $N_3>N_2$),
$$
\Big|\int_{\RR^d} \xi(x)(\rho_n-\rho)(dx)\Big| \leq \frac \eps 4.
$$
Thus, for $n\geq N_3$
$$
\rho_n\big(B(x,\frac 1n)\setminus\{x\}\big) \leq \frac \eps 2.
$$
Plugging this latter inequality into \eqref{ineq3} and from \eqref{ineq2},
we deduce that for $n\geq N_3$,
$$
\Big|\int_{x\neq y} \big(\nabla W_n(x-y)-\nabla W(x-y)\big) \rho_n(dy)\Big| \leq
\frac 1n + C \frac{\eps}{2}
$$
Finally, combining this latter inequality with \eqref{ineq1}, we deduce that
for $n\geq N_3$,
$$
\big|a_n(x)-a(x)| \leq \frac\eps 4 + \frac 1n + C\frac \eps 2,
\qquad \mbox{for a.e. }x\in \RR^d.
$$
\end{proof}

\bigskip
{\bf Acknowledgements.}
JAC acknowledges support from projects MTM2011-27739-C04-02,
2009-SGR-345 from Ag\`encia de Gesti\'o d'Ajuts Universitaris i de Recerca-Generalitat 
de Catalunya, the Royal Society through a Wolfson Research Merit Award, and the Engineering and Physical Sciences Research Council (UK) grant number EP/K008404/1.
NV acknowledges partial support from the french "ANR blanche" project Kibord~: ANR-13-BS01-0004.

\bigskip

%%%%%%%%%%%%%%%%%%%%%%%%%%%%%%%%%%%
%
%%%%%% BIBLIO %%%%%%%%%%%%%%%%%%%%%%
%
%%%%%%%%%%%%%%%%%%%%%%%%%%%%%%%%%%%%

% \bibliography{aggregation.bib}
% \bibliographystyle{plain}
%%%%%%%%%%%%%%%%%%%%%%%%%%%%%%%%%%%%%%%%%%%%%%%%%%%

%%%%%%%%%%%%%%%%%%%%%%%%%%%%%%%%%%%%%%%%%%%%%%%%%%%%%%%%%%%%%%%%%%%%%%%%%%%%%%%%

\newpage

\begin{center}
  {\Large Erratum: The Filippov characteristic flow for the aggregation equation with mildly singular potentials}  \\[5mm]

  José Antonio Carrillo, François James, Frédéric Lagoutière, \\ David Poyato, Nicolas Vauchelet
\end{center}

\vspace{1cm}

\section{Introduction}

In this erratum we provide a corrected version and a corrected proof of an existence and uniqueness result in \cite{CJLV} concerning weak measure-valued solutions to the so-called aggregation equation in space dimension $d$. The original statement of the theorem containing the mistake is reminded in Theorem \ref{wrongtheo}, and its corrected version is stated in Theorem \ref{Exist}. The aggregation equation reads
\begin{equation}\label{EqInter}
\begin{aligned}
&\pa_t\rho = \dv\big((\nabla_x W*\rho) \rho\big), \quad t>0,\, x\in\RR^d,\\
&\rho(0,\cdot)=\rho^{ini},
\end{aligned}
\end{equation}
for some initial condition $\rho(0,\cdot)=\rho^{ini}$.
In this equation, $W$ is an interaction potential whose gradient $\nabla_x W(x-y)$ measures the relative effect exerted by a unit mass localized at a point $y$ onto the velocity of a unit mass located at a point $x$.
As in \cite{CJLV}, we assume that the interaction potential $W\,:\,\RR^d\to\RR$ is \textit{pointy}, i.e. it satisfies the following properties:
\begin{itemize}
\item[{\bf (A0)}] $W$ is Lipschitz-continuous, $W(x)=W(-x)$ and $W(0)=0$;
\item[{\bf (A1)}] $W$ is $\lambda$-convex for some 
{$\lambda  \leq 0$}, i.e.  $W(x)-\frac{\lambda}{2}|x|^2$ is convex;
\item[{\bf (A2)}] $W\in C^1(\RR^d\setminus\{0\})$.
\end{itemize}
Typical examples are fully attractive potentials $W(x)=1-e^{-|x|}$, or $W(x) = |x|$.
Notice that the Lipschitz-continuity of the potential allows to bound the velocity field: 
\begin{equation}
\label{borngradW}
\exists\, w_\infty > 0: \quad \|\nabla W \|_\infty \leq w_\infty.
\end{equation}

We denote by $C_0(\RR^d)$ the space of continuous functions from $\RR^d$ to $\RR$ that tend to $0$ at $\infty$, and 
$\calM_b(\RR^d)$ the space of Borel signed measures whose total variation is finite.
We call $\calP(\RR^d)$ the subset of $\calM_b(\RR^d)$ of probability measures, and $\calP_2(\RR^d)$ the subset of probability measures with finite
second order moment.
The space ${\mathcal P}_{2}(\RR^d)$ is equipped with the Wasserstein distance $d_W$ defined by (see e.g. \cite{Ambrosio, Filippo c touo})
%\beq\label{defWp}
$$
d_W(\mu,\nu) := \inf_{\gamma\in \Gamma(\mu,\nu)} \left\{\int_{\RR^d\times \RR^d} |y-x|^2\,\gamma(dx,dy)\right\}^{1/2},
$$
%\eeq
where $\Gamma(\mu,\nu)$ is the set of measures on $\RR^d\times\RR^d$ with marginals $\mu$ and $\nu$.

\subsection{Definition of the velocity field and its Filippov's flow}
The aggregation equation \eqref{EqInter} can be regarded as a continuity equation whose velocity field is determined by the convolution $-\nabla_xW*\rho$. In view {\bf (A0)}-{\bf (A2)}, we remark that $\nabla W$ could be discontinuous and its value is not well defined at $x=0$. Therefore, the above convolution is weakly defined and setting a precise pointwise definition of the velocity field is crucial. Specifically, given any such curve $\rho \in C([0,+\infty),\calP_2(\RR^d))$, we shall define its associated velocity field $\widehat{a}_{\rho}$ by
\beq\label{achapo}
\achapo_{\rho}(t,x)
:=
 -\int_{\RR^d} \nabWchapo(x-y) \rho(t,dy),\quad t\geq 0,\,x\in \RR^d,
\eeq
where we have used the notation
$$
\nabWchapo(x) := \left\{
\begin{array}{ll} \nabla W(x), \qquad & \mbox{ for } x\neq 0, 
\\
0, & \mbox{ for } x=0. \end{array}
\right.
$$

On the one hand, due to the Lipschitz-continuity of $W$, see {\bf (A0)}, which implies \eqref{borngradW} as mentioned above, we obtain the following uniform bound for the velocity field $\achapo_\rho$ defined in \eqref{achapo}
\begin{equation}\label{eq:abound}
|\achapo_\rho(t,x)|\leq w_\infty,\quad x\in\RR^d,\,t\geq0.
\end{equation}
On the other hand, due to the $\lambda$-convexity of $W$, see {\bf (A1)}, we deduce
\beq\label{lambdaconv}
\langle \nabla W(x)-\nabla W(y) , x-y\rangle \geq \lambda |x-y|^2,\quad x,y\in \RR^d\setminus\{0\}.
\eeq
This is not enough to ensure the Lipschitz-continuity of the velocity field $\achapo_\rho$ in \eqref{achapo}, which is in fact discontinuous at the atoms of the probability measure $\rho(t)$, but it is clear that \eqref{lambdaconv} implies the following one-sided Lipschitz estimate for $\achapo_\rho$
\begin{equation}
\label{eq:aOSL}
\bigl\langle \achapo_{\rho}(t,x)-\achapo_{\rho}(t,y), x-y\bigr\rangle \leq -\lambda |x-y|^2,\quad  t \geq 0,\,x,y\in\RR^d.
\end{equation}

By virtue of the uniform bound \eqref{eq:abound} and the one-sided Lipschitz estimate \eqref{eq:aOSL}, we may define a Filippov characteristic flow for the velocity field $\achapo_\rho$ which is globally-in-time defined, and also unique forward-in-time, see \cite{Filippov}. Specifically, for every time $s\geq 0$ and each point $x\in \RR^d$ there exists a unique absolutely continuous solution $Z_\rho(t;s,x)$ to the following differential inclusion
\begin{equation}\label{eq:characteristicsFilippov}
\begin{aligned}
&\frac{d}{dt} Z_\rho(t;s,x)\in [\achapo_\rho(t,\cdot)](Z_\rho(t;s,x)),\quad \mbox{a.e.}\ t\geq 0,\\
&Z_\rho(s;s,x)=x.
\end{aligned}
\end{equation}
Above, the notation $[\achapo_\rho(t,\cdot)]$ stands for the essential convex hull (also called {\it Filippov convexification}) of the bounded and measurable velocity field $\achapo_\rho(t,\cdot):\RR^d\longrightarrow\RR^d$. It is defined as follows
\begin{equation}\label{essconvexhull}
[\achapo_\rho(t,\cdot)](x):=\bigcap_{r>0}\bigcap_{N\in \mathcal{N}_0}\co(\achapo_\rho(t,B(x,r)\setminus N)),\quad t\geq 0,\,x\in \RR^d,
\end{equation}
where $\mathcal{N}_0$ is the set of zero Lebesgue measure sets, and $\co(A)$ denotes the closed convex hull of any set $A\subset \RR^d$. In particular, the Filippov characteristic verifies that $Z_\rho(\cdot;s,x)\in C([s,+\infty),\RR^d)$, it is differentiable at almost every $t\geq s$, and it satisfies the differential inclusion almost everywhere. From now on, we will make use of the shorthand $Z_\rho(t,x)=Z_\rho(t;0,x)$ to simplify our notation.

%{\color{red} {\bf To be erased:} For $s\geq 0$ and $x\in \RR^d$, we will denote by  $Z(\cdot;s,x)\in C([s,+\infty);\RR^d)$ a Filippov characteristic starting from $x$ at time $s$ for the velocity field $\achapo_\rho$.
%We recall the semi-group property: for any $t,\tau,s \in [0,+\infty)$ such that $t \geq \tau \geq s$ and $x\in \RR^d$,
%\begin{equation}
%\label{eq:characteristics}
%Z(t;s,x)=Z(\tau;s,x)+\int_{\tau}^t \achapo_{\rho}\bigl(\sigma,Z(\sigma;s,x)\bigr)\,d\sigma.
%\end{equation}
%}

\subsection{The issue: Non-uniqueness of solutions defined by Filippov's flow}

Using the above notion of Filippov characteristic flow, it has been established in \cite{PoupaudRascle} that, for any given bounded and one-sided Lipschitz velocity field, the solutions to corresponding (linear) conservative transport equation can be defined as the pushforward of the initial condition along the Filippov characteristic flow of the velocity field. Based on this approach, existence and uniqueness of solutions to the (nonlinear) aggregation equation \eqref{EqInter} defined by a Filippov flow had been established in \cite{CJLV}. More precisely, the following result was stated in \cite[Theorem 2.5]{CJLV}.

\begin{theorem}[Original version]\label{wrongtheo}
Let $W$ satisfy assumptions {\bf (A0)}-{\bf (A2)} and let $\rho^{ini}$ be given in $\mathcal{P}_2(\RR^d)$. Given any $T>0$, there exists a unique Filippov characteristic flow $Z$ such that the pushforward measure $\rho:=Z_{\#}\rho^{ini}$ is a distributional solution of the aggregation equation
$$
\begin{aligned}
&\partial_t\rho +\dv(\achapo_\rho\rho)=0, \quad t>0,\, x\in\RR^d,\\
&\rho(0,\cdot)=\rho^{ini},
\end{aligned}
$$
where $\achapo_\rho$ is defined by \eqref{achapo}. 

Besides, if $\rho^{ini}$ and $\mu^{ini}$ are two given nonnegative measures in $\mathcal{P}_2(\RR^d)$, then the corresponding pushforward measures $\rho$ and $\mu$ satisfy for all $t\in [0,T]$
$$d_W(\rho(t),\mu(t))\leq e^{-2\lambda t}d_W(\rho^{ini},\mu^{ini}).$$
\end{theorem}

We have identified a mistake in the uniqueness part of the proof of Theorem \ref{wrongtheo} given in \cite[Theorem 2.5]{CJLV}. Specifically, it is still true (see corrected version in Theorem \ref{Exist}) that there exists a unique distributional solution of the above aggregation equation, and additionally all distributional solutions of the above aggregation equation are solutions defined by Filippov's flow, that is, $\rho(t)=Z_\rho(t,\cdot)_\#\rho^{ini}$ where $Z_\rho$ is the unique Filippov's characteristic flow associated to $\achapo_\rho$, {\it cf.} \eqref{eq:characteristicsFilippov}. However, in general it is false that the later type of solutions (i.e., solutions defined by Filippov's flow) amount to the former type of solutions (i.e., distributional solutions). Specifically, there is non-uniqueness of solutions defined by Filippov's flow, as we discuss below.

\begin{remark}[Non-uniqueness of solutions defined by Filippov's flow]
Consider the problem of finding solutions defined by Filippov's flow issued at $\rho^{ini}\in \mathcal{P}_2(\RR^d)$, i.e., curves of probability measures $\rho\in C([0,+\infty),\mathcal{P}_2(\RR^d))$ such that
\[
\left\{\begin{array}{l}
\rho(t):=Z_{\rho}(t,\cdot){}_\# \rho^{ini}, \\
\frac{d}{dt} Z_\rho(t;s,x)\in [\achapo_\rho(t,\cdot)](Z_\rho(t;s,x)),\quad \mbox{a.e.}\ t\geq 0,\\
Z_\rho(s;s,x)=x.
\end{array}\right.
\]
where, $[\widehat{a}_\rho (t,\cdot)]$ denotes the essential convex hull of $\widehat{a}_\rho (t,\cdot)$ introduced in \eqref{essconvexhull}. We show below that given an initial datum $\rho^{ini}$ such a problem could admit more than one solution. 

Indeed, for dimension $d = 1$ and potential $W(x) = |x|$, consider the initial datum $\rho^{ini} = \delta_0$. On the one hand, it is straightforward to check that $\rho_1(t) = \delta_0$ is a solution to this problem (and even, it will be the unique distributional solution to the aggregation equation as per Theorem \ref{Exist}). Specifically, note that its associated velocity field $\achapo_{\rho_1}$ and its essential convex hull have the form
$$\achapo_{\rho_1}(t,x)=\left\{\begin{array}{l}
1, \quad x < 0, \\
0, \quad x = 0, \\
-1, \quad x > 0, 
\end{array}\right.\quad [\widehat{a}_{\rho_1}(t,x)] = \left\{\begin{array}{l}
\{1\}, \quad x < 0, \\
\lbrack -1, 1 \rbrack, \quad x = 0, \\
\{-1\}, \quad x > 0.
\end{array}\right.$$
Hence, the unique Filippov characteristic of $\achapo_{\rho_1}$ starting at $x=0$ has the form $Z_{\rho_1}(t,0)=0$ because
$$\frac{d}{dt}Z_{\rho_1}(t,0)=0\in [-1,1]=[\achapo_{\rho_1}(t,\cdot)](Z_{\rho_1}(t,0)),\quad t\geq 0,$$
which implies that $Z_{\rho_1}(t,\cdot)_{\#}\rho^{ini}=\rho_1(t)$ for all $t\geq 0$. On the other hand, we also have that $\rho_2(t) = \delta_t$ defines a second solution to the above problem issued at the same initial datum $\rho^{ini}$. This time, the velocity field induced by $\rho_2$ and its essential convex hull read
\[
\widehat{a}_{\rho_2}(t,x) = \left\{\begin{array}{l}
1, \quad x < t, \\
0, \quad x = t, \\
-1, \quad x > t, 
\end{array}\right.\quad [\widehat{a}_{\rho_2}(t,x)] = \left\{\begin{array}{l}
\{1\}, \quad x < t, \\
\lbrack -1, 1 \rbrack, \quad x = t, \\
\{-1\}, \quad x > t.
\end{array}\right.
\]
Thereby, the unique Filippov characteristic of $\achapo_{\rho_2}$ starting at $x=0$ has the form $Z_{\rho_2}(t,0) = t$ because 
$$\frac{d}{dt}Z_{\rho_2}(t,0)=1\in [-1,1]=[\achapo_{\rho_2}(t,\cdot)](Z_{\rho_2}(t,0)),\quad t\geq 0,$$
which implies that $Z_{\rho_2}(t,\cdot)_{\#}\rho^{ini}=\rho_2(t)$ for all $t\geq 0$.

Nevertheless, there is a big difference between both choices of Filippov's flow: whilst in fact $Z_{\rho_1}(t,0)=0$ solves the characteristic system in the classical sense
$$\frac{d}{dt}Z_{\rho_1}(t,0)=\achapo_{\rho_1}(t,Z_{\rho_1}(t,0)),\quad t\geq 0,$$
the second curve $Z_{\rho_2}(t,0) = t$ only verifies the characteristic system in Filippov's sense
$$\frac{d}{dt}Z_{\rho_2}(t,0)\in [\achapo_{\rho_2}(t,\cdot)](Z_{\rho_2}(t,0)), \quad t\geq 0.$$
This difference is crucial and yields completely different behaved solutions: whilst $\rho_1$ is a distributional solution to the aggregation equation, $\rho_2$ is not a distributional solution. Here, and contrary to the linear setting studied in \cite{PoupaudRascle}, the nonlinearity of the problem under consideration is responsible for this gap. Nevertheless, as we will see in Theorem \ref{Exist}, given any distributional solution $\rho$ to the aggregation equation, and once its velocity field $\achapo_\rho$ is computed, it is a matter of fact that $\rho(t)=Z_{\rho}(t,\cdot)_{\#}\rho^{ini}$, where $Z_\rho$ is the unique Filippov flow of the characteristic equation. This is why we still chose to call this solution a Filippov-type solution. However, verifying the characteristic system in the classical sense is fundamental in order to have uniqueness and stability results.

%{\color{red}{\bf To be erased}: However notice that $Z_{\rho_2}(t,0) = t$ of course does not satisfy the integral equation \eqref{eq:characteristics}. 
%This makes a difference with the classical theory of Filippov, in which the integral form and the 
%differential inclusion are equivalent (for absolutely continuous solutions). Here the nonlinearity of the problem under 
%consideration is responsible for this gap. Nevertheless once the velocity field is computed (with \eqref{achapo}) it is a matter of fact that $\rho(t)$ is the pushforward of $\rho^{ini}$ by the unique Filippov flow of the characteristic equation, this is why we chose to call this solution a Filippov solution. The integral form of the characteristics \eqref{eq:characteristics} is fundamental.}

\end{remark}

\subsection{The solution: Uniqueness of distributional solutions}

The {\em corrected} version of \cite[Theorem 2.5]{CJLV} reads as follows:

\begin{theorem}[Corrected version]\label{Exist}
Let $W$ satisfy assumptions {\bf (A0)--(A2)} and let $\rho^{ini}$ be given in $\calP_2(\RR^d)$. Then, there exists a unique distributional solution $\rho \in C([0,+\infty),\calP_2(\RR^d))$ of
\begin{equation}
\label{eq:agreg:TH}
\begin{aligned}
&\partial_t\rho +\dv(\achapo_\rho\rho)=0, \quad t>0,\, x\in\RR^d,\\
&\rho(0,\cdot)=\rho^{ini},
\end{aligned}
\end{equation}
where $\achapo_{\rho}$ is defined by \eqref{achapo}. This unique distributional solution may be represented in the following form as the family of pushforward measures 
\begin{equation}\label{eq:push-forward}
\rho(t):=Z_{\rho}(t,\cdot)_\# \rho^{ini},\quad t\geq 0,
\end{equation}
where $Z_{\rho}$ is the unique Filippov characteristic flow associated to the velocity field $\achapo_{\rho}$, {\it cf.} \eqref{eq:characteristicsFilippov}. Additionally, for $\rho^{ini}$-a.e. $x\in \RR^d$ the Filippov characteristic $Z_\rho(\cdot,x)$ is actually a classical solution verifying the characteristic system in integral form, i.e.,
%$$\frac{d}{dt}Z_\rho(t,x)=\achapo_\rho(t,Z_\rho(t,x)),\quad \rho^{ini}\mbox{-a.e.}\ x\in \RR^d,\,\mbox{a.e.}\ t\geq 0,$$
\begin{equation}
\label{eq:characteristics}
Z_\rho(t,x)=x+\int_0^t \achapo_{\rho}\bigl(s,Z_\rho(s,x)\bigr)\,ds,\quad \rho^{ini}\mbox{-a.e.}\ x\in \RR^d,\,\mbox{a.e.}\ t\geq 0.
\end{equation}

Besides, if $\rho$ and  $\rho'$ are the respective distributional solutions of \eqref{eq:agreg:TH} with $\rho^{ini}$ and $\rho^{ini,\prime}$ as initial conditions in ${\mathcal P}_{2}(\RR^d)$,  then 
\begin{equation}
  \label{eq:contract}
d_{W}(\rho(t),\rho'(t)) \leq e^{-\lambda t}
d_{W}(\rho^{ini},\rho^{ini,\prime}), \qquad t \geq 0.    
\end{equation}
\end{theorem}

\begin{remark}
The result relies strongly on the precise definition \eqref{achapo} of the velocity field $\achapo_\rho$. As mentioned above, we remark that at any point $x\in \RR^d$ where $\rho(t)$ has an atom, the field $\widehat{a}_\rho(t,\cdot)$ is discontinuous. Hence, defining the value of $\widehat{a}_\rho(t,x)$ at all $(t,x)$ (not only almost everywhere) is a way to define properly the ambiguous product $\widehat{a}_\rho \rho$ in the definition of distributional solutions to \eqref{eq:agreg:TH}.

The uniqueness part also relies strongly on the fact that $\widehat{a}_\rho(t,\cdot)$ satisfies the one-sided Lipschitz estimate \eqref{eq:aOSL}, which makes the push-forward representation formula \eqref{eq:push-forward} valid for all distributional solutions and, more interestingly, for $\rho^{ini}$-a.e. initial datum $x\in \RR^d$ the trajectory $Z_\rho(\cdot,x)$ can be formulated by a classical characteristic system \eqref{eq:characteristics} instead of a differential inclusion \eqref{eq:characteristicsFilippov}. At all other $x\in \RR^d$ away of the support of $\rho^{ini}$ the trajectory may be understood in Filippov's sense yet. This will be crucially used in the new proof of the stability estimate \eqref{eq:contract} of distributional solutions.

Finally, we remark that a classical formulation \eqref{eq:characteristics} is not available for the solutions in Filippov's sense of a general abstract ODE system unless we modify the definition of the velocity field on a suitable negligible set (see \cite[Theorem 3.5]{PoupaudRascle}), which we cannot do in our nonlinear setting since $\achapo_\rho$ must be defined at all points by \eqref{achapo}. Fortunately, \eqref{eq:characteristics} is at least valid for the Filippov flow of the velocity field $\achapo_\rho$ associated to a distributional solution $\rho\in C([0,+\infty),\mathcal{P}_2(\mathbb{R}^d))$ to \eqref{eq:agreg:TH}.
\end{remark}

In the next section, we provide a proof of this theorem. On the one hand, the original proof of existence presented in \cite[Theorem 2.4]{CJLV} has been completed. On the other hand, the proof of uniqueness has been corrected, using directly the integral formula \eqref{eq:characteristics} for the Filippov flow any distributional solutions instead of the regularization process in \cite[Theorem 2.4]{CJLV}, which contains a mistake. The main changes in the proof of this result are detailed in the following section.

\section{Proof of Theorem \ref{Exist}}
\label{sec:proof}

The proof of existence is based on the idea of atomization, consisting in approximating the distributional solution by a finite sum of Dirac masses (or particles), and then passing to the limit. The proof of uniqueness relies strongly on the stability estimate \eqref{eq:contract}, which in turn yields uniqueness of distributional solutions. This latter estimate exploits in a crucial way that general distributional solutions of the aggregation equation \eqref{eq:agreg:TH} can be represented by the push-forward formula \eqref{eq:push-forward}, and also that the new integral formulation of the characteristic system \eqref{eq:characteristics} holds. This section is organized as follows. We shall start by proving existence in Section \ref{subsec:proof-existence}, then the representation and integral formulas in Section \ref{subsec:proof-integral-formula}, and finally the stability estimate and uniqueness in Section \ref{subsec:proof-stability}.

\subsection{Proof of existence}\label{subsec:proof-existence}

\subsubsection{Approximation with Dirac masses}

Let us assume that the initial density is given by
$\rho^{ini,N}(x) = \sum_{i=1}^N m_i \delta(x-x_i^0)$ for
a finite integer $N$, with $x_i^0\neq x_j^0$ for $i\neq j$ and $\sum_{i=1}^N m_i = 1$, thus belonging to $\calP_2(\RR^d)$. %, i.e. we have
%\beq\label{hyprhoini}
%\sum_{i=1}^N m_i = 1, \qquad M_2(0):=\sum_{i=1}^N m_i |x_i^0|^2  <+\infty.
%\eeq
Following \cite{CJLV}, the goal is to look for $\rho^N(t,x) = \sum_{i=1}^N m_i \delta(x-x_i(t))$ solving the aggregation equation \eqref{eq:agreg:TH} in distributional sense. This suggests that positions $x_1,\ldots,x_N$ should solve the ODE system
\begin{equation}\label{ODEsyst}
\begin{aligned}
&x'_i(t) = -\sum_{j=1}^N m_j \widehat{\nabla W}(x_i(t)-x_j(t)),\\
&x_i(0) = x_i^0,  \quad i=1,\ldots,N.
\end{aligned}
\end{equation}

Let us define $t\to X(t) = (x_1(t), \ldots, x_N(t))^\top \in \RR^{Nd}$. The above dynamical system may be rewritten
$X'(t) = F(X(t))$, with the vector field $F:\RR^{Nd} \to \RR^{Nd}$ defined by
$$F((x_1,\ldots,x_N)^\top) = - \bigg(\sum_{j=1}^N m_j \widehat{\nabla W}(x_1(t)-x_j(t)),\ldots,\sum_{j=1}^N m_j \widehat{\nabla W}(x_N(t)-x_j(t))\bigg)^\top.$$
We verify that $F$ satisfies a one-sided Lipschitz condition for the weighted inner product $\langle \cdot,\cdot\rangle_m$ on $\RR^{Nd}$ defined by $\langle u,v\rangle_m := \sum_{k=1}^N m_k \langle u_k,v_k\rangle$, with $\langle\cdot,\cdot\rangle$ the usual inner product in $\RR^d$. We compute
\begin{align*}
  \langle F(X)-F(Y),X-Y\rangle_m \ 
  & = -\sum_{k=1}^N m_k \sum_{j=1}^N m_j \langle \widehat{\nabla W}(x_k-x_j)-\widehat{\nabla W} (y_k-y_j),x_k-y_k\rangle  \\
  & = -\frac 12 \sum_{k,j=1}^N m_k m_j  \langle \widehat{\nabla W}(x_k-x_j)-\widehat{\nabla W} (y_k-y_j),x_k-x_j-y_k+y_j\rangle,  
\end{align*}
thanks to the symmetry of $W$, thus 
\[
\langle F(X)-F(Y),X-Y\rangle_m  
\leq -\frac{\lambda}{2} \sum_{k,j=1}^N m_k m_j |x_k-y_k-x_j+y_j|^2,
\]
where we use the $\lambda$-convexity assumption {\bf (A1)} of $W$ ({\it cf.} \eqref{lambdaconv}) for the last inequality.
Hence,
\begin{align*}
  \langle F(X)-F(Y),X-Y\rangle_m \leq -2\lambda \sum_{k=1}^N m_k |x_k-y_k|^2 = -2\lambda |X-Y|^2_m,
\end{align*}
(recall that $\lambda \leq 0$ and that $\sum_{j=1}^N m_j = 1$). Additionally, since $\nabla W$ is bounded ({\it cf.} \eqref{borngradW}) by the Lipschitz-continuity assumption {\bf (A0)} of $W$, we also have that $F$ satisfies the uniform bound
$$
|F(X)|_m\leq w_\infty,\quad X\in \RR^{dN}.
$$
Then, again from the Filippov theory \cite{Filippov}, there exists a unique global-in-time Filippov solution $X$ to the system \eqref{ODEsyst}, which is understood as an absolutely continuous solution to the differential inclusion into $X'(t)\in [F](X(t))$ for a.e. $t\geq 0$, {\it cf.} \eqref{essconvexhull}.

We remark though that such a unique Filippov solution must actually solve the differential equation $X'(t)=F(X(t))$ for a.e. $t\geq 0$ in the classical sense. Indeed, since we depart from a non-collisional initial datum (i.e., $x_i^0\neq x_j^0$ for all $i\neq j$), then the Filippov solution $X$ is understood in the classical sense until it eventually breaks down at some finite time $t^*$, at which some particles collide. We claim that Filippov's dynamics selects a continuation of the classical solution by sticking of the groups formed after the collision time, see \cite[\S 3.2]{PPS} for more details. Specifically, for all $i\in \{1,\ldots,N\}$ take the subset $J_i\subset \{1,\ldots,N\}$ of particles colliding with particle $i$, that is, $x_i(t^*) = x_j(t^*)$ for $j\in J_i$, but $x_i(t_*)\neq x_j(t_*)$ for $j\notin J_i$. Assume that exactly $M$ groups of particles (with $M<N$) get formed at $t^*$ with indices $i_1,\ldots,i_M\in \{1,\ldots,N\}$. Define the reduced system consisting of $M$ particles $y_1,\ldots, y_M$ evolving according to
$$
\begin{aligned}
&y'_k(t) = -\sum_{l=1}^M n_l \widehat{\nabla W}(y_k(t)-y_l(t)),\quad t\geq t^*,\\
&y_i(t^*) = x_{i_k}(t^*),  \quad k=1,\ldots,M,
\end{aligned}
$$
with $n_k:=\sum_{j\in J_{i_k}}m_j$ the total mass on each group. Since the initial data of the reduced system is non-collisional by definition, a classical solution exists and extends until a new eventual collision time $t^{**}>t^*$. By uniqueness of the Filippov solution to \eqref{ODEsyst} we infer that $x_j(t)=y_k(t)$ for $t\in [t^*,t^{**}]$, all $j\in J_{i_k}$ and all $k=1,\ldots,M$. Since $\widehat{\nabla W}(0)=0$, this amounts to saying that the Filippov solution $X$ solves the differential equation \eqref{ODEsyst} in the classical sense also in $[t^*,t^{**}]$. Repeating this procedure finitely many times, we cover the full lifespan of the Filippov solution.

Having $X$ solving \eqref{ODEsyst}, we define its associated curve of probability measures
$$\rho^N(t,x) = \sum_{i = 1}^N m_i \delta(x - x_i(t)),$$
whose velocity field $\widehat{a}_{\rho^N}$ in \eqref{achapo} then has the form: 
\[
\widehat{a}_{\rho^N}(t,x) = -\sum_{j=1}^N m_j \widehat{\nabla W}(x-x_j(t)).
\]
This velocity field satisfies again the uniform bound \eqref{eq:abound} and the one-sided Lipschitz estimate \eqref{eq:aOSL}, which allows defining a global-in-time unique Filippov flow $\Zchapo^N$, see \cite{Filippov}. Next, we define $\rho_{PR}^N = \Zchapo^N_{\ \#} \rho^{ini,N}$. Thanks to \cite{PoupaudRascle}, $\rho_{PR}^N$ must solve in distributional sense the transport equation
$$
\pa_t \rho_{PR}^N + \dv \big(\achapo_{\rho^N} \rho_{PR}^N\big) = 0.
$$
Moreover, from the definition of the pushforward measure, we can write
$$
\achapo_{\rho_{PR}^N} = -\int_{\RR^d} \nabWchapo(x-y) \rho_{PR}^N(dy)
= - \int_{\RR^d} \nabWchapo(x-\Zchapo^N(t,y)) \rho^{ini,N}(dy).
$$
By definition of $\rho^{ini,N}$, we deduce
$$
\begin{array}{ll}
\achapo_{\rho_{PR}^N}(t,x) &\ds = - \sum_{i=1}^N m_i \nabWchapo(x-\Zchapo^N(t,x_i^0))
= \achapo_{\rho^N}(t,x).
\end{array}
$$
Thus we conclude that $\rho_{PR}^N=\rho^N$, and therefore $\rho^N$ solves the aggregation equation \eqref{eq:agreg:TH} in distributional sense with initial data $\rho^{ini,N}$.

Since we have the bound \eqref{borngradW} on $\nabla W$, we again obtain the bound \eqref{eq:abound} on $\achapo_{\rho^N}$, that is
\beq\label{boundachapo}
|\achapo_{\rho^N}(t,x)|\leq w_\infty,\quad t\geq 0,\,x\in \RR^d.
\eeq
Arguing as in \cite{CJLV}, the above implies that the second order moment is bounded uniformly on each time interval $[0,T]$, then $\rho^N\in C([0,T],\calP_2(\RR^d))$ for all $T>0$, and therefore $\rho^N\in C([0,+\infty),\mathcal{P}_2(\mathbb{R}^d))$.

\subsubsection{Passing to the limit $N\to +\infty$}

For $\rho^{ini}\in \calP_2(\RR^d)$, we consider an approximation $\rho^{ini,N}\in \calP_2(\RR^d)$ given by a finite sum of Dirac masses such that $d_W(\rho^{ini,N},\rho^{ini})\to 0$ as $N\to\infty$. In particular, $\rho^{ini,N} \rightharpoonup \rho^{ini}$ weakly in the sense of measures in $\mathcal{M}_b(\RR^d)$. In the previous section, we have proved that we can construct a Filippov flow $\Zchapo^N$ and a measure $\rho^N = \Zchapo^N\,_{\#} \rho^{ini,N}\in C([0,+\infty),\calP_2(\RR^d))$ solving the aggregation equation
\begin{equation}\label{eq:rhoN}
\pa_t \rho^N + \dv (\achapo_{\rho^N} \rho^N) = 0,
\end{equation}
in distributional sense, where $\achapo_{\rho^N}$ is defined by \eqref{achapo}.
From \eqref{boundachapo}, we have that
$\achapo_{\rho^N}$ is bounded in $L^\infty([0,T]\times\RR^d)$.
Thus $\achapo_{\rho^N}$ converges up to a subsequence towards $b$
in $L^\infty_{t,x}-weak*$.  Passing to the limit in the distributional sense in the uniform bound \eqref{eq:abound} and the one-sided Lipschitz estimate \eqref{eq:aOSL} (both holding with $N$-independent parameters $w_\infty>0$ and $\lambda<0$), we deduce that $b$ belongs to $L^\infty([0,T]\times \RR^d)$ and satisfies the one-sided Lipschitz condition. Then, we can define $Z_b$ the global-in-time unique Filippov flow corresponding to $b$. From the $L^\infty_{t,x}-weak*$ convergence above, it is obvious that $\achapo_{\rho^N}$ converges weakly to $b$ in $L^1([0,T];L^1_{loc}(\RR^d))$. Therefore, we can apply the stability result in \cite[Theorem 1.2]{Bianchini}
and deduce that $\Zchapo^N \to Z_b$ locally in $C([0,T]\times \RR^d)$ as $N\to +\infty$.

Moreover, it has been proved in \cite[\S 3.2.2]{CJLV} that for every $\phi \in C_0(\RR^d)$, we have
$$
\int_{\RR^d} \phi(x) \rho^N(t,dx)
= \int_{\RR^d} \phi(\Zchapo^N(t,x)) \rho^{ini,N}(dx) \underset{N \to +\infty}{\longrightarrow}
\int_{\RR^d} \phi(X_b(t,x)) \rho^{ini}(dx),
$$
uniformly in $t\in [0,T]$. Indeed, as above, the uniform estimate \eqref{boundachapo} implies that second order moments of $\rho^N(t)$ are uniformly bounded with respect to $t\in [0,T]$ and $N\in \mathbb{N}$. Then, a standard cut-off argument ensures that the above must also holds for every $\phi \in C_b(\RR^d)$. Therefore, we deduce that $\rho^N \to \rho:=X_b\,_\# \rho^{ini}$ in $C([0,T],\mathcal{P}(\RR^d)\mbox{-narrow})$ as $N\to +\infty$. From this latter convergence, we deduce by applying \cite[Lemma A.1]{CJLV}
that $\achapo_{\rho^N} \to \achapo_{\rho}$ a.e., which implies that $b=\achapo_{\rho}$ a.e.

Finally, from \cite[Lemma 3.1]{CJLV}, we have that $\achapo_{\rho^N} \rho^N \rightharpoonup \achapo_\rho \rho$.
As a consequence, we can pass to the limit in the sense of distributions in the equation \eqref{eq:rhoN}, and we deduce that $\rho\in C([0,T],\mathcal{P}(\RR^d)\mbox{-narrow})$ is a distributional solution of \eqref{EqInter}. The bound of the second order moment of $\rho(t)$ is similar to the proof in \cite[\S 3.2.3]{CJLV}, and follows by the lower semicontinuity of the integrals with respect to the narrow convergence. We leave the proof of the fact that actually $\rho\in C([0,T],\mathcal{P}_2(\RR^d))$ to next Section.

%\poyato{I think we still need to prove that $\rho\in C([0,+\infty),\mathcal{P}_2(\RR^d))$. For Ambrosio's superposition theorem (used in next section) to operate on this type of distributional solution we need at least that $\rho\in C([0,+\infty),\mathcal{P}(\RR^d)\mbox{-narrow})$.}

\subsection{Proof of representation and integral formulas}\label{subsec:proof-integral-formula}
Consider a  distributional solution $\rho\in C([0,+\infty),\mathcal{P}(\RR^d)\mbox{-narrow})$ to the aggregation equation \eqref{eq:agreg:TH} with initial datum $\rho^{ini}\in \mathcal{P}_2(\RR^d)$, as in the previous Section \ref{subsec:proof-existence}. For any $T>0$, we have that $\rho\in C([0,T],\mathcal{P}(\RR^d)\mbox{-narrow})$, and the uniform bound \eqref{eq:abound} of $\achapo_\rho$ implies that
$$\int_0^T\int_{\RR^d}|\achapo_\rho(t,x)|^2\,\rho(t,dx)\,dt\leq Tw_\infty^2<\infty.$$
Hence, we can use the probabilistic representation in \cite[Theorem 8.2]{Ambrosio}, \cite[Theorem 4.4]{AmbrosioCrippa}. Specifically, there is a probability measure $\eta\in \mathcal{P}(\RR^d\times \Gamma_T)$, where $\Gamma_T:=C([0,T],\RR^d)$ is endowed with the uniform norm, such that $\eta$ is concentrated on pairs $(x,\gamma)$ with $x\in\RR^d$ and $\gamma\in AC^2(0,T;\RR^d)$ solving
\begin{equation}\label{eq:cauchyproblem}
\begin{aligned}
&\gamma'(t)=\achapo_\rho(t,\gamma(t)), \quad \mbox{a.e.}\ t\in [0,T],\\
&\gamma(0)=x,
\end{aligned}
\end{equation}
and such that $\rho(t):=e_{t\#}\eta$ for every $t\in [0,T]$, where the mapping $e_t:\RR^d\times \Gamma_T\longrightarrow \RR^d$ stands for the evaluation map defined by $e_t(x,\gamma)=\gamma(t)$, i.e.,
\begin{equation}\label{eq:probrepresentation}
\int_{\RR^d}\phi(x)\,\rho(t,dx)=\int_{\RR^d\times \Gamma_T}\phi(\gamma(t))\,\eta(dx,d\gamma),
\end{equation}
for all $\phi\in C_b(\RR^d)$. Evaluating \eqref{eq:probrepresentation} at $t=0$, and using that $\rho(0)=\rho^{ini}$ and $\gamma(0)=x$ for $\eta$-a.e. $(x,\gamma)\in \RR^d\times \Gamma_T$, we have $\pi_{x\#}\eta=\rho^{ini}$. Therefore, disintegrating $\eta$ with respect to $x$ yields a Borel family of probability measures $(\eta^x)_{x\in \RR^d}\subset \mathcal{P}(\Gamma_T)$ such that 
$$\eta(dx,d\gamma)=\rho^{ini}(dx)\otimes \eta^x(d\gamma),$$
see \cite[Theorem 5.3.1]{Ambrosio}. Using the above disintegration in \eqref{eq:probrepresentation} we infer
\begin{equation}\label{eq:probrepresentation2}
\int_{\RR^d}\phi(x)\,\rho(t,dx)=\int_{\RR^d}\int_{\Gamma_T}\phi(\gamma(t))\,\eta^x(d\gamma)\,\rho^{ini}(dx),
\end{equation}
for all $\phi\in C_b(\RR^d)$. Since $\eta$ is supported on the pairs $(x,\gamma)$ with $x\in \RR^d$ and $\gamma\in AC^2(0,T;\RR^d)$ solving \eqref{eq:cauchyproblem}, we deduce that for $\rho^{ini}$-a.e. $x\in \RR^d$ the conditional probability measure $\eta^x$ must be supported on the curves $\gamma\in AC^2(0,T;\RR^d)$ solving \eqref{eq:cauchyproblem}. We remark that $\eta^x$ is indeed a probability measure, and then it cannot have an empty support. Therefore, for $\rho^{ini}$-a.e. $x\in \RR^d$ the Cauchy problem \eqref{eq:cauchyproblem} must have at least one classical solution $\gamma_x\in AC^2(0,T;\RR^d)$. Since $\achapo_\rho$ verifies the one-sided Lipschitz estimate \eqref{eq:aOSL}, actually only one such classical solution exists and by uniqueness of the Filippov flow we indeed deduce $\gamma_x=Z_\rho(\cdot,x)$, which further implies that \eqref{eq:characteristics} holds true and that $\eta^x=\delta_{Z_\rho(\cdot,x)}$. Therefore, evaluating the integrals with respect to $\eta^x$ in \eqref{eq:probrepresentation2} yields
$$
\int_{\RR^d}\phi(x)\,\rho(t,dx)=\int_{\RR^d}\phi(Z_\rho(t,x))\rho^{ini}(dx),
$$
for all $\phi\in C_b(\RR^d)$, that is, the push-forward representation $\rho(t)=Z_\rho(t,\cdot)_{\#}\rho^{ini}$ in \eqref{eq:push-forward} is verified.

We finally prove that not only $\rho\in C([0,+\infty),\mathcal{P}(\RR^d)\mbox{-narrow})$ but also $\rho\in C([0,+\infty),\mathcal{P}_2(\RR^d))$. By the push-forward representation \eqref{eq:push-forward} above we obtain the relation
\begin{align*}
\int_{\RR^d}|x|^2\,\rho(t+h,dx)&-\int_{\RR^d}|x|^2\,\rho(t,dx)\\
&=\int_{\RR^d}(|Z_\rho(t+h,x)|^2-|Z_\rho(t,x)|^2)\,\rho^{ini}(dx)\\
&=2\int_t^{t+h}\int_{\RR^d}\langle Z_\rho(s,x),\achapo_\rho(s,x)\rangle\,\rho^{ini}(dx)\,ds,
\end{align*}
which by Jensen's inequality implies
$$\left\vert\int_{\RR^d}|x|^2\,\rho(t+h,dx)-\int_{\RR^d}|x|^2\,\rho(t,dx)\right\vert\leq 2w_\infty\int_t^{t+h}\left(\int_{\RR^d}|x|^2\,\rho(s,dx)\right)^{1/2}\,ds.$$
By the uniform bound \eqref{eq:abound} of $\achapo_\rho$, it is easy to prove that $|Z_\rho(t,x)|^2\lesssim |x|+t$ for $\rho^{ini}$-a.e. $x\in \RR^d$, and therefore the second order moments of $\rho(t)$ are bounded uniformly on each $[0,T]$. Hence, we can pass to the limit above as $h\to 0$ above and conclude that $\rho\in C([0,+\infty),\mathcal{P}_2(\RR^d))$.

\subsection{Proof of uniqueness}\label{subsec:proof-stability}

We start by proving the stability estimate \eqref{eq:contract} of distributional solutions to the aggregation equation \eqref{eq:agreg:TH}. Consider any couple $\rho,\rho'\in C([0,+\infty),\calP_2(\RR^d))$ of distributional solutions with respective initial conditions $\rho^{ini},\rho^{ini,\prime}\in \mathcal{P}_2(\RR^d)$. Their related velocity fields $\achapo_\rho$ and $\achapo_{\rho'}$ defined via \eqref{achapo} both satisfy the uniform bound \eqref{eq:abound} and the one-sided Lipschitz estimate \eqref{eq:aOSL}. Thus there exists a unique Filippov flow $Z_\rho$ and $Z_{\rho\prime}$ associated to each vector field, and by the push-forward representation \eqref{eq:push-forward} obtained in the above Section \ref{subsec:proof-integral-formula} we also have
%have one and only one solution. As a direct consequence of Theorem 8.2.1 in \cite[Chapter 8.2]{Ambrosio} (see also \cite[Section 4]{AmbrosioCrippa}), 
%$\rho$ and $\rho'$ can be expressed as the pushforwards of $\rho^{ini}$ and $\rho^{ini,\prime}$ by the associated flows, that is to say that with the Fillipov flows $(Z_{\rho}(t,\cdot))_{t \geq 0}$
%and $(Z_{\rho'}(t,\cdot))_{t \geq 0}$ defined in the statement of 
%Theorem \ref{Exist} we have 
\begin{equation}
\label{eq:aux:proof}
Z_{\rho}(t,\cdot){}_{\#} \rho^{ini}=
\rho(t,\cdot), \qquad 
Z_{\rho'}(t,\cdot){}_{\#} \rho^{ini,\prime}=
\rho'(t,\cdot), \qquad t \geq 0.
\end{equation}
Also, as proven in Section \ref{subsec:proof-integral-formula}, the integral formula \eqref{eq:characteristics} holds for both Filippov flows for $\rho^{ini}$-a.e. (respectively $\rho^{ini,\prime}$-a.e.) initial datum, that is,
\begin{equation}
\label{eq:aux2:proof}
\begin{aligned}
&Z_\rho(t,x)=x+\int_0^t \achapo_{\rho}\bigl(s,Z_\rho(s,x)\bigr)\,ds,\quad \rho^{ini}\mbox{-a.e.}\ x\in \RR^d,\,\mbox{a.e.}\ t\geq 0,\\
&Z_{\rho\prime}(t,y)=y+\int_0^t \achapo_{\rho\prime}\bigl(s,Z_{\rho\prime}(s,y)\bigr)\,ds,\quad \rho^{ini,\prime}\mbox{-a.e.}\ y\in \RR^d,\,\mbox{a.e.}\ t\geq 0.
\end{aligned}
\end{equation}

To simplify the notations, in the sequel we just write 
$Z(t,\cdot) = Z_{\rho}(t,\cdot)$
and
$Z'(t,\cdot) = Z_{\rho'}(t,\cdot)$.
Then, from \eqref{eq:aux2:proof}, we have
$$
Z(t,x)-Z'(t,y) = x-y + \int_0^t (\achapo_\rho(s,Z(s,x))-\achapo_{\rho'}(s,Z'(s,y)))\,ds,
$$
for $\rho^{ini}$-a.e. $x\in \RR^d$ and $\rho^{ini,\prime}$-a.e. $y\in \RR^d$. 
Thus
\begin{align*}
|Z(t,x)-Z'(t,y)|^2 = & |x-y|^2 + \Big|\int_0^t \big(\achapo_\rho(s,Z(s,x))-\achapo_{\rho'}(s,Z'(s,y))\big)\,ds\Big|^2   \\
& + 2\int_0^t \langle x-y, \achapo_\rho(s,Z(s,x))-\achapo_{\rho'}(s,Z'(s,y))\rangle\,ds  \\[2mm]
= & |x-y|^2 + \Big|\int_0^t \big(\achapo_\rho(s,Z(s,x))-\achapo_{\rho'}(s,Z'(s,y))\big)\,ds\Big|^2 \\
& + 2\int_0^t \langle Z(s,x)-Z'(s,y), \achapo_\rho(s,Z(s,x))-\achapo_{\rho'}(s,Z'(s,y))\rangle\,ds  \\
& + 2\int_0^t \langle x-Z(s,x)+Z'(s,y)-y, \achapo_\rho(s,Z(s,x))-\achapo_{\rho'}(s,Z'(s,y))\rangle\,ds,
\end{align*}
for $\rho^{ini}$-a.e. $x\in \RR^d$ and $\rho^{ini,\prime}$-a.e. $y\in \RR^d$. By definition \eqref{eq:aux2:proof}, we may rewrite the last term 
\begin{align*}
2\int_0^t &\langle  x-Z(s,x)+ Z'(s,y)-y \,, \,\achapo_\rho(s,Z(s,x))-\achapo_{\rho'}(s,Z'(s,y))\rangle\,ds \\
         & = 2\int_0^t\int_0^s \langle \achapo_{\rho'}(\tau,Z'(\tau,y))-\achapo_\rho(\tau,Z(\tau,x)), \achapo_\rho(s,Z(s,x))-\achapo_{\rho'}(s,Z'(s,y))\rangle\,d\tau\,ds  \\
         & = - 2 \Big|\int_0^t \big(\achapo_\rho(s,Z(s,x))-\achapo_{\rho'}(s,Z'(s,y))\big)\,ds\Big|^2  \\
         & \quad - 2 \int_0^t\int_s^t \langle \achapo_{\rho'}(\tau,Z'(\tau,y))-\achapo_\rho(\tau,Z(\tau,x)), \achapo_\rho(s,Z(s,x))-\achapo_{\rho'}(s,Z'(s,y))\rangle\,d\tau\,ds.
\end{align*}
Using Fubini's theorem, we also obtain
\begin{align*}
2\int_0^t &\langle  x-Z(s,x)+ Z'(s,y)-y \,, \,\achapo_\rho(s,Z(s,x))-\achapo_{\rho'}(s,Z'(s,y))\rangle\,ds \\
& = - 2 \Big|\int_0^t \big(\achapo_\rho(s,Z(s,x))-\achapo_{\rho'}(s,Z'(s,y))\big)\,ds\Big|^2  \\
& \quad - 2 \int_0^t\int_0^\tau \langle \achapo_{\rho'}(\tau,Z'(\tau,y))-\achapo_\rho(\tau,Z(\tau,x)), \achapo_\rho(s,Z(s,x))-\achapo_{\rho'}(s,Z'(s,y))\rangle\,ds\,d\tau.
\end{align*}
Hence,
\begin{align*}
  2 \int_0^t \langle  x-Z(s,x) + Z'(s,y)-y \,, \,\achapo_\rho(s,Z(s,x))-\achapo_{\rho'}(s,Z'(s,y))\rangle\,ds & \\
  = - \Big|\int_0^t \big(\achapo_\rho(s,Z(s,x))-\achapo_{\rho'}(s,Z'(s,y))\big)\,ds\Big|^2. &
\end{align*}
We the arrive at
\begin{equation}\label{eqXY1}
|Z(t,x) - Z'(t,y)|^2 =  |x - y|^2 +
2\int_0^t \langle Z(s,x)-Z'(s,y), \achapo_\rho(s,Z(s,x))-\achapo_{\rho'}(s,Z'(s,y))\rangle\,ds,
\end{equation}
for $\rho^{ini}$-a.e. $x\in \RR^d$ and $\rho^{ini,\prime}$-a.e. $y\in \RR^d$. 

Set any optimal plan $\pi \in \Gamma_{0}(\rho^{ini},\rho^{ini,\prime})$ between $\rho^{ini}$ and $\rho^{ini,\prime}$. 
Then, integrating \eqref{eqXY1} with respect to $\pi$, which can be done because the identity holds except on a $\pi$-negligible set,  we deduce
\begin{equation}\label{eqinter2}
\int_{\RR^d}\int_{\RR^d} |Z(t,x) - Z'(t,y)|^2\,\pi(dx,dy) = \int_{\RR^d}\int_{\RR^d} |x-y|^2\,\pi(dx,dy) + I,
\end{equation}
where we define
\[
I = 2\int_0^t \int_{\RR^d} \int_{\RR^d} \langle Z(s,x)-Z'(s,y), \achapo_\rho(s,Z(s,x))-\achapo_{\rho'}(s,Z'(s,y))\rangle\,\pi(dx,dy)\,ds.
\]
As $\rho(s) = Z(s,\cdot) \# \rho^{ini}$ and $\rho^\prime(s) = Z^\prime(s,\cdot) \# \rho^{ini,\prime}$ by \eqref{eq:aux:proof}, the definition \eqref{achapo} implies
\begin{align*}
\widehat{a}_\rho(s,Z(s,x)) & %= - \int_{\RR^d} \widehat{\nabla W}(Z(s,x) - x')\,\rho(s,dx')
=- \int_{\RR^d} \widehat{\nabla W}(Z(s,x) - Z(s,x'))\,\rho^{ini}(dx')\\
&=-\iint_{\RR^d\times \RR^d}\widehat{\nabla W}(Z(s,x)-Z(s,x'))\,\pi(dx',dy'),
\end{align*}
and similarly
\begin{align*}
\widehat{a}_{\rho^\prime}(s,Z^\prime(s,y)) & %= - \int_{\RR^d} \widehat{\nabla W}(Z^\prime(s,y) - y')\,\rho^\prime(s,dy')
=- \int_{\RR^d} \widehat{\nabla W}(Z^\prime(s,y) - Z^\prime(s,y'))\,\rho^{ini,\prime}(dy')\\
&=-\iint_{\RR^d\times \RR^d}\widehat{\nabla W}(Z^\prime(s,y)-Z^\prime(s,y'))\,\pi(dx',dy').
\end{align*}
Therefore we can write
\begin{align*}
I = & - 2\int_0^t \iint_{\RR^d\times\RR^d}\iint_{\RR^d\times\RR^d} \big\langle Z(s,x)-Z'(s,y), \nabWchapo(Z(s,x)-Z(s,x'))-\nabWchapo(Z'(s,y)-Z'(s,y'))\big\rangle  \\
&\hspace{11cm}\,\pi(dx,dy)\,\pi(dx',dy')\,ds \\
= & - \int_0^t \iint_{\RR^d\times\RR^d}\iint_{\RR^d\times\RR^d} \big\langle Z(s,x)-Z(s,x')-Z'(s,y)+Z'(s,y'), \\
&\hspace{3.5cm} \nabWchapo(Z(s,x)-Z(s,x'))-\nabWchapo(Z'(s,y)-Z'(s,y'))\big\rangle
\,\pi(dx,dy)\,\pi(dx',dy')\,ds,
\end{align*}
where we exchanged the role of $x,y$ with $x',y'$ and used the symmetry assumption on $W$ in assumption {\bf (A0)} to obtain the last equality.
By $\lambda$-convexity of $W$ {\bf (A1)}, we deduce
\begin{align*}
I &\leq - \lambda \int_0^t \iint_{\RR^d\times\RR^d}\iint_{\RR^d\times\RR^d} \big| Z(s,x)-Z(s,x')-Z'(s,y)+Z'(s,y')\big|^2\,\pi(dx,dy)\,\pi(dx',dy')\,ds  \\
& \leq - 2\lambda \int_0^t\left(\iint_{\RR^d\times\RR^d}|Z(s,x)-Z'(s,y)|^2\,\pi(dx,dy) - \left\vert\iint_{\RR^d\times\RR^d} (Z(s,x)-Z'(s,y))\,\pi(dx,dy)\right\vert^2\right) \,ds.
\end{align*}
Thus, recalling that $\lambda\leq 0$, we have obtained from \eqref{eqinter2}
\begin{align*}
\iint_{\RR^d\times\RR^d} |Z(t,x) - Z'(t,y)|^2\,\pi(dx,dy)
\leq & \iint_{\RR^d\times\RR^d} |x - y|^2\,\pi(dx,dy)  \\
& - 2\lambda \int_0^t \iint_{\RR^d\times\RR^d} |Z(s,x) - Z'(s,y)|^2\,\pi(dx,dy)\,ds.
\end{align*}
Thanks to Gr{\"o}nwall's lemma, we deduce
$$
\iint_{\RR^d\times\RR^d} |Z(t,x) - Z'(t,y)|^2\,\pi(dx,dy)
\leq d_{W}(\rho^{ini},\rho^{ini,\prime})^2  e^{-2\lambda t},
$$
where we use the fact that
$\iint_{\RR^d\times\RR^d}\vert x - y \vert^2 \pi(dx,dy) = d_{W}(\rho^{ini},\rho^{ini,\prime})^2$ by definition of optimal plan.
We conclude by noticing that $\pi_t:=(Z(t,\cdot)\otimes Z'(t,\cdot))_{\#}\pi\in \Gamma(\rho(t),\rho'(t))$ by \eqref{eq:aux:proof} and therefore
\begin{equation*}
 \int_{\RR^d}  \int_{\RR^d}\vert Z(t,x) - Z'(t,y) \vert^2
\,\pi(dx,dy)=\int_{\RR^d}  \int_{\RR^d}\vert x-y \vert^2
\,\pi_t(dx,dy) \geq 
d_{W}(\rho(t),\rho(t)^{\prime})^2,
\end{equation*}
by definition of the Wasserstein distace.

Uniqueness is deduced from the stability estimate in Wasserstein distance as proved above. Indeed, if we take $\rho^{ini} = \rho^{ini,\prime}$ in the stability estimate \eqref{eq:contract}, then we deduce that $\rho = \rho'$.

We finally remark that this uniqueness proof is reminiscent of the computations to characterize the element of minimal norm in the subdifferential of the interaction energy used in \cite{CDFLS} to construct unique solutions to the aggregation equations via the JKO approach.

\bigskip

%%%%%%%%%%%%%%%%%%%%%%%%%%%%%%%%%%%
%
%%%%%% BIBLIO %%%%%%%%%%%%%%%%%%%%%%
%
%%%%%%%%%%%%%%%%%%%%%%%%%%%%%%%%%%%%

% \bibliography{aggregation.bib}
% \bibliographystyle{plain}
%%%%%%%%%%%%%%%%%%%%%%%%%%%%%%%%%%%%%%%%%%%%%%%%%%%

%%%%%%%%%%%%%%%%%%%%%%%%%%%%%%%%%%%%%%%%%%%%%%%%%%%%%%%%%%%%%%%%%%%%%%%%%%%%%%%%

\end{document}